\documentclass[a4ppaper,14pt]{amsart}

\usepackage{enumerate}
\usepackage[latin1]{inputenc}
\usepackage{amsmath}
\usepackage{amsthm}
\usepackage{latexsym}
\usepackage{amssymb}
\usepackage{amsmath}
\usepackage{comment}
\usepackage[all]{xy}
\usepackage{calc}
\usepackage{multicol}
\usepackage{slashed}

\usepackage{amsfonts}
\usepackage{graphicx}
\usepackage{amscd}

\usepackage{color}

\newtheorem{thm}{Theorem}[section]

\newtheorem{defi}{Definition}[section]
\newtheorem{lem}{Lemma}[section]

\theoremstyle{notation}
\newtheorem*{notation}{Notation}

\newcommand{\R}{\mathbb{R}}
\newcommand{\C}{\mathbb{C}}
\numberwithin{equation}{section}
\newcommand{\N}{\mathbb{N}}
\newcommand{\Z}{\mathbb{Z}}

\newcommand{\eps}{\epsilon}

\newcommand{\wto}{\rightharpoonup}

\newcommand{\vertiii}[1]{{\left\vert\kern-0.25ex\left\vert\kern-0.25ex\left\vert #1
\right\vert\kern-0.25ex\right\vert\kern-0.25ex\right\vert}}

\usepackage[textwidth=138mm,
textheight=215mm,
left=30mm,
right=30mm,
top=25.4mm,
bottom=25.4mm,
headheight=2.17cm,
headsep=4mm,
footskip=12mm,
heightrounded,]{geometry}

\usepackage[colorlinks=true,urlcolor=blue,
citecolor=black,linkcolor=black,linktocpage,pdfpagelabels,
bookmarksnumbered,bookmarksopen]{hyperref}

\makeatletter
\newcommand{\leqnomode}{\tagsleft@true}
\newcommand{\reqnomode}{\tagsleft@false}
\makeatother

\begin{document}

\reqnomode

\title{Existence and dynamics of normalized solutions to nonlinear Schr\"{o}dinger equations with mixed fractional Laplacians}
\author[L. Chergui, T. Gou, H. Hajaiej]{Lassaad Chergui, Tianxiang Gou, Hichem Hajaiej}

\address{Lassaad Chergui,
\newline \indent Department of Mathematics, College of Science and Arts in Uglat Asugour,
\newline \indent Qassim University,  Buraydah, Kingdom of Saudia Arabia.
\newline \indent Department of Mathematics, College of Science in Bizerte,
\newline \indent 7021 Jarzouna, Carthage University, Tunis, Tunisia.}
\email{L.CHERGUI@qu.edu.sa}

\address{Tianxiang Gou
\newline \indent School of Mathematics and Statistics, Xi'an Jiaotong University,
\newline \indent Xi'an, Shaanxi 710049, People's Republic of China.}
\email{tianxiang.gou@xjtu.edu.cn}

\address{Hajaiej Hichem,
\newline \indent Department of Mathematics, College of Natural Sciences, California State University,
\newline \indent 5151 State Drive, 90032 Los Angeles, California, USA.}
\email{hhajaie@calstatela.edu}

\thanks{ {\bf Data Availability Statements}: The manuscript has no associated data.}
\thanks{{Acknowlegments}: Some of the results developed in the part dealing with the Cauchy problem in this paper have been used in L. Chergui's very recently accepted paper \cite{LC}. He unintentionally did not cite this paper.}

\begin{abstract} In this paper,
we are concerned with the existence and dynamics of solutions to the equation with mixed fractional Laplacians
$$
(-\Delta)^{s_1} u +(-\Delta)^{s_2} u + \lambda u=|u|^{p-2} u
$$
under the constraint
$$
\int_{\R^N} |u|^2 \, dx=c>0,
$$
where $N \geq 1$, $0<s_2<s_1<1$, $2+ \frac {4s_1}{N} \leq p< \infty $ if $N \leq 2s_1$, $2+ \frac {4s_1}{N} \leq p<\frac{2N}{N-2s_1}$ if $N >2s_1$, $\lambda \in \R$ appearing as Lagrange multiplier is unknown. The fractional Laplacian $(-\Delta)^s$ is characterized as $\mathcal{F}((-\Delta)^{s}u)(\xi)=|\xi|^{2s} \mathcal{F}(u)(\xi)$ for $\xi \in \R^N$, where $\mathcal{F}$ denotes the Fourier transform.  First we establish the existence of ground state solutions and the multiplicity of bound state solutions. Then we study dynamics of solutions to the Cauchy problem for the associated time-dependent equation. Moreover, we establish orbital instability of ground state solutions.



\end{abstract}

\maketitle

\tableofcontents

\section{Introduction and main results}

In this paper, we are interested in the existence and dynamics of solutions to the following equation with mixed fractional Laplacians,
\begin{align} \label{fequ}
(-\Delta)^{s_1} u +(-\Delta)^{s_2} u + \lambda u=|u|^{p-2} u,
\end{align}
under the constraint
\begin{align} \label{mass}
\int_{\R^N} |u|^2 \, dx=c>0,
\end{align}
where $N \geq 1$, $0<s_2<s_1<1$, $2+ \frac {4s_1}{N} \leq p< \infty $ if $N \leq 2s_1$, $2+ \frac {4s_1}{N} \leq p< \frac{2N}{N-2s_1}$ if $N >2s_1$, $\lambda \in \R$ appearing as Lagrange multiplier is unknown. The fractional Laplacian $(-\Delta)^s$ is characterized as $\mathcal{F}((-\Delta)^{s}u)(\xi)=|\xi|^{2s} \mathcal{F}(u)(\xi)$ for $\xi \in \R^N$, where $\mathcal{F}$ denotes the Fourier transform. The equation \eqref{fequ} arises from the study of standing waves to the time-dependent equation
\begin{align}\label{evolv pb0}
\left\{
\begin{aligned}
&i\partial_t \psi -(-\Delta)^{s_1} \psi -(-\Delta)^{s_2}\psi =-|\psi|^{p-2}\psi, \\
&\psi(0,x)=\psi_0(x), \quad x \in \R^N,
\end{aligned}
\right.
\end{align}
where $N\geq 1$, $0<s_2<s_1<1$, $2+ \frac {4s_1}{N} \leq p< \infty $ if $N \leq 2s_1$ and $2+ \frac {4s_1}{N} \leq p<\frac{2N}{N-2s_1}$ if $N >2s_1$. Here standing waves to \eqref{evolv pb0} are solutions of the form
$$
\psi(t, x)=e^{i\lambda t} u(x), \quad \lambda \in \R.
$$
It is obvious to see that standing wave $\psi$ is a solution to \eqref{evolv pb0} if and only if $u$ is a solution to \eqref{fequ}.

The equation \eqref{evolv pb0} appears in many fields and has been increasingly attracting the attention of scientists in recent years due to its numerous and important applications. It models many biological phenomena like describing the diffusion in an ecological niche subject to nonlocal dispersals. In the niche, the population is following a certain process so that if an individual exist the niche, it must come to the niche right away by selecting the return point according to the underlying stochastic process. This results in an equation involving mixed fractional Laplacians. The mixed operators are the outcome of the superposition of two long-range L\'evy processes or a classical Brownian motion and a long-range process. The population diffuses according to two or more types of nonlocal dispersals, modeled by L\'evy flights and encoded by two or more fractional Laplacians with two different powers, see \cite{VS} for more detailed accounts. The sum and the difference of two or more fractional Laplacians appear in many other fields, we refer the reader to page 2 of \cite{CBH} and the references therein for more details.

Note that any solution $\psi \in C([0, T), H^{s_1}(\R^N))$ to \eqref{evolv pb0} conserves the mass along time, i.e.
$$
\|\psi(t)\|_2=\|\psi_0\|_2, \quad \forall\,\, t \in [0, T).
$$
The mass often admits a clear physical meaning, for instance it represents the power supply in nonlinear optics or the total number of atoms in Bose-Einstein condensation. Therefore, from a physical point of view, it is interesting to explore standing waves to \eqref{evolv pb0} with prescribed $L^2$-norm. This then leads to the study of solutions to \eqref{fequ}-\eqref{mass}. Such solutions are often called normalized solutions to \eqref{fequ}. In this scenario, the parameter $\lambda \in \R$ is unknown and to be determined as Lagrange multiplier. Here we shall focus on normalized solutions to \eqref{fequ}. It is standard to check that any solution $u \in H^{s_1}(\R^N)$ to \eqref{fequ}-\eqref{mass} corresponds to a critical point of the functional
$$
E(u):=\frac{1}{2} \int_{\R^N} |(-\Delta)^{\frac{s_1}{2}} u|^2\,dx + \frac{1}{2} \int_{\R^N} |(-\Delta)^{\frac{s_2}{2}} u|^2\,dx-\frac{1}{p} \int_{\R^N}|u|^p\,dx
$$
restricted on the constraint
$$
S(c):=\left\{u \in H^{s_1}(\R^N) : \int_{\R^N}|u|^2 dx =c\right\}.
$$

When $2<p<2 + \frac{4s_1}{N}$, by using Gagliardo-Nirenberg inequality \eqref{gn}, we find that $E$ restricted on $S(c)$ is bounded from below for any $c>0$. Therefore, we are able to introduce the following minimization problem,
\begin{align} \label{gmin111}
m(c):=\inf_{u \in S(c)} E(u).
\end{align}
Apparently, minimizers to \eqref{gmin111} are solutions to \eqref{fequ}-\eqref{mass}. In this case, the authors in \cite{HL} established the existence of minimizers to \eqref{gmin111}. However, when $p \geq 2 + \frac{4s_1}{N}$, the study of solutions to \eqref{fequ}-\eqref{mass} is open so far. The aim of the present paper is to make some contributions towards this direction.

Firstly, we shall consider the existence of solutions to \eqref{fequ}-\eqref{mass} for the case $p=2 + \frac{4s_1}{N}$. In this case, by utilizing Gagliardo-Nirenberg inequality \eqref{gn}, we have the following result.

\begin{thm} \label{thm1}
Let $N \geq 1$, $0<s_2<s_1<1$ and $p=2+\frac{4s_1}{N}$. Then there exists a constant $c_{N, s_1}>0$ such that
\begin{align*}
m(c)=\left\{
\begin{aligned}
&0, \qquad 0<&c\leq c_{N, s_1},\\
&-\infty, \qquad & c>c_{N, s_1}.
\end{aligned}
\right.
\end{align*}
In addition, for any $0<c \leq  c_{N, s_1}$, $m(c)$ is not attained and there exists no solutions to \eqref{fequ}-\eqref{mass}, where $c_{N,s_1}>0$ is given by
$$
c_{N, s_1}:=\left(\frac{N+2s_1}{N C_{N, s_1}}\right)^{\frac{N}{2s_1}}
$$
and $C_{N, s_1}=C_{N,p,s_1}>0$ is the optimal constant in \eqref{gn} for $p=2+\frac{4s_1}{N}$.
\end{thm}

From Theorem \ref{thm1}, we see that $E$ restricted on $S(c)$ is unbounded from below for any $c>c_{N,s_1}$. This then suggests that it is unlikely to take advantage of \eqref{gmin111} to seek for solutions to \eqref{fequ}-\eqref{mass} for any $c>c_{N,s_1}$. This is also the case when $p>2+\frac{4s_1}{N}$. Indeed, for any $u \in S(c)$ and $t >0$, we define
$$
u_t(x):=t^{\frac N 2} u(tx), \quad x \in \R.
$$
By straightforward calculations, then $\|u_t\|_2=\|u\|_2$ and
\begin{align}\label{scaling}
E(u_t)=\frac{t^{2s_1}}{2} \int_{\R^N} |(-\Delta)^{\frac{s_1}{2}} u|^2\,dx + \frac{t^{2s_2}}{2} \int_{\R^N} |(-\Delta)^{\frac{s_2}{2}} u|^2\,dx-\frac{t^{\frac{N}{2}(p-2)}}{p} \int_{\R^N}|u|^p\,dx,
\end{align}
from which we conclude that $E(u_t) \to -\infty$ as $t \to \infty$, because of $p>2+\frac{4s_1}{N}$. Then there holds that $m(c)=-\infty$ for any $c>0$. In such a situation, deriving the existence of solutions to \eqref{fequ}-\eqref{mass}, we need to introduce the following minimization problem,
\begin{align} \label{min}
\gamma(c):=\inf_{u \in P(c)} E(u),
\end{align}
where $P(c)$ is the so-called Pohozaev manifold defined by
$$
P(c):=\{u \in S(c) : Q(u)=0\}
$$
and
$$
Q(u):=\frac{d}{dt}E(u_t)\mid_{t=1}=s_1\int_{\R^N} |(-\Delta)^{\frac{s_1}{2}} u|^2\,dx + s_2 \int_{\R^N} |(-\Delta)^{\frac{s_2}{2}} u|^2\,dx-\frac{N(p-2)}{2p} \int_{\R^N}|u|^p\,dx.
$$
Here $Q(u)=0$ is the Pohozaev identity associated to solutions of \eqref{fequ}-\eqref{mass}, see Lemma \ref{pohozaev}.

To further state the existence results for the case $p \geq 2+\frac{4s_1}{N}$, we define a constant $c_0 \geq 0$ by $c_0=c_{N, s_1}$ if $p=2+\frac{4s_1}{N}$ and $c_0=0$ if $p>2+\frac{4s_1}{N}$, where $c_{N,s_1}>0$ is the constant determined in Theorem \ref{thm1}

\begin{thm}\label{thm2}
Let $N \geq 1$, $0<s_2<s_1<1$ and $p \geq 2+\frac{4s_1}{N}$. Then there exists a constant $c_1>c_0$ such that, for any $c_0<c<c_1$, \eqref{fequ}-\eqref{mass} has a ground state solution $u_c \in S(c)$ satisfying $E(u_c)=\gamma(c)$. In particular, if $N=1$ and $2s_2 \geq 1$ or $N \geq 1$, $2s_2<N$ and $2< p \leq \frac{2N}{N-2s_2}$, then $c_1=\infty$
\end{thm}

To prove Theorem \ref{thm2}, the essential argument is to demonstrate that $P(c)$ is a natural constraint, by which we can obtain a Palais-Smale sequence belonging to $P(c)$ for $E$ restricted on $S(c)$ at the level $\gamma(c)$ for any $c>c_0$. Later, by using the fact that $E$ restricted on $P(c)$ is coercive, then the Palais-Smale sequence is bounded in $H^{s_1}(\R^N)$. Finally, by verifying that the function $c \mapsto \gamma(c)$ is nonincreasing on $(c_0, \infty)$ and the associated Lagrange multiplier $\lambda_c$ is positive for any $c_0<c<c_1$, then the compactness of the Palais-Smale sequence in $H^{s_1}(\R^N)$ follows. This completes the proof.

\begin{thm} \label{thm6}
Let $N \geq 1$, $0<s_2<s_1<1$ and $p \geq 2+\frac{4s_1}{N}$. If $u \in S(c)$ is a ground state solution to \eqref{fequ}-\eqref{mass} at the level $\gamma(c)$, then $u$ admits the form $e^{i \theta} |u_c|$ for some $\theta \in \mathbb{S}^1$, where $|u_c| \geq 0$ is radially symmetric and nonincreasing up to translations.
\end{thm}

The proof of Theorem \ref{thm6} is principally based on the variational characteristics of ground state solutions to \eqref{fequ}-\eqref{mass} and P\'olya-Szeg\"o inequality for fractional Laplacian.

\begin{thm} \label{thm3}
Let $N \geq 2$ and $0<s_2<s_1<1$.
\begin{enumerate}
\item [$(\textnormal{i})$] If $p>2+\frac{4s_1}{N}$, then, for any $0<c<c_1$, \eqref{fequ}-\eqref{mass} has infinitely many radially symmetric solutions $\{u_k\} \subset H^{s_1}(\R^N)$ satisfying $E(u_{k+1}) \geq E(u_k)>0$ and $E(u_k) \to \infty$ as $k \to \infty$, where $c_1>0$ is the constant determined in Theorem \ref{thm2}.
\item [$(\textnormal{ii})$] If $p=2+\frac{4s_1}{N} \leq \frac{2N}{N-2s_2}$, then, for any $k \in \N^+$, there exists a constant $c_k>c_{N,s_1}$ such that, for any $c>c_k$, \eqref{fequ}-\eqref{mass} has at least $k$ radially symmetric solutions in $H^{s_1}(\R^N)$.
\end{enumerate}
\end{thm}

To achieve Theorem \ref{thm3}, we shall work in the subspace $H^{s_1}_{rad}(\R^N)$ consisting of radially symmetric functions in $H^{s_1}(\R^N)$. 
By applying the Kranosel'skii genus theory and following the strategies of the proof of Theorem \ref{thm2}, we can complete the proof. It is worth mentioning that the discussion of the compactness of Palais-Smale sequence for $E$ restricted on $S(c)$ becomes somewhat simple in $H^{s_1}_{rad}(\R^N)$, because the embedding $H^{s_1}_{rad}(\R^N) \hookrightarrow L^p(\R^N)$ is compact for any $2<p<\frac{2N}{N-2s_1}$ and $N \geq 2$.


\begin{thm}\label{thm4}
Let $N \geq 1$, $0<s_2<s_1<1$ and $p \geq 2+\frac{4s_1}{N}$.
\begin{enumerate}
\item[$(\textnormal{i})$] The function $c \mapsto \gamma(c)$ is continuous for any $c>c_0$ and it is nonincreasing on $(0, \infty)$. Moreover, $\lim_{c \to c_0^+} \gamma(c)=\infty$.
\item[$(\textnormal{ii})$] The function $c \mapsto \gamma(c)$ is strictly decreasing on $(c_0, c_1)$. Moreover, if $N=1$ and $2s_2 \geq 1$ or $N \geq 1$, $2s_2<N$ and $2 < p < \frac{2N}{N-2s_2}$, then the function $c \mapsto \gamma(c)$ is strictly decreasing on $(0, \infty)$
and $\lim_{c \to \infty} \gamma(c)=0$.
\item[$(\textnormal{iii})$] If $N>\max\left\{2 s_1+2,\frac{2s_1s_2}{s_1-s_2}\right\}$, then there exists a constant $c_{\infty}>0$ such that $\gamma(c)=m$ for any $c \geq c_{\infty}$, where $m>0$ is the ground state energy to \eqref{fequ00}.
\end{enumerate}
\end{thm}

The proofs of the assertions $(\textnormal{i})$ and $(\textnormal{ii})$ of Theorem \ref{thm4} are primarily beneficial from the definition of $\gamma(c)$. To prove the assertion $(\textnormal{iii})$ of Theorem \ref{thm4}, we first need to establish the existence of ground state solutions to the zero mass equation
\begin{align}  \label{fequ00}
(-\Delta)^{s_1} u +(-\Delta)^{s_2} u =|u|^{p-2} u
\end{align}
in a proper Sobolev space $H$ defined by the completion of $C^{\infty}_0(\R^N)$ under the norm
$$
\|u\|_H:=\left(\int_{\R^N}|(-\Delta)^{\frac{s_1}{2}} u|^2 \,dx\right)^{\frac 12}+\left(\int_{\R^N}|(-\Delta)^{\frac{s_2}{2}} u|^2 \,dx\right)^{\frac 12}.
$$
Then we require to show that the solutions belong to $L^2(\R^N)$, see Lemma \ref{l2}.

Let us now mention a few related works with respect to the study of normalized solutions to various nonlinear Schr\"odinger-type equations and systems. For the mass subcritical case, by the well-known Gagliardo-Nirenberg inequality, one derives that the energy functionals restricted on the $L^2$-norm constraints are bounded from below. In this situation, by introducing global minimization problems as the energy functionals restricted on the constraints, one can consider the existence and orbital stability of normalized solutions in the spirit of the Lions concentration compactness principle \cite{Li1, Li}, see for example \cite{AB, CS, CCW, CDSS, CP, G, Gou, GJ2, NW1, NW2, NW3, S} and references therein. Here normalized solutions corresponds to global minimizers.

For the mass critical or supercritical cases, things become quite different and complex. In these cases, the energy functionals restricted on the $L^2$-norm constraints may be unbounded from below, then it is impossible to bring in global minimization problems to investigate the existence of normalized solutions. In this situation, normalized solutions often corresponds to saddle type critical points or local minimizers, the existence of which are guaranteed by minimax arguments. For a long time, the paper \cite{Je} due to Jeanjean is the only one dealing with the existence of normalized solutions when the energy functionals restricted on the constraints are unbounded from below. During recent years,  because of its physical relevance and mathematical importance in theories and applications, the study of normalized solutions has received more attention from researchers, see for example \cite{BMRV, BJS, BS1, BS2, BV, BZZ, BJ, BJT, BCGJ, CJ, GJ1, GZ, HT, JS, JL, LY, NTV2, S1, S2} regarding normalized solutions to equations and systems in $\R^N$ and \cite{NTV1, NTV3, PPVV, PG} regarding normalized solutions to equations and systems in bounded domains.

Now we turn to investigate dynamics of solutions to the Cauchy problem for the time-dependent equation \eqref{evolv pb0}. To do this, we first need to establish the local wellposdness of solutions in $H^{s_1}(\R^N)$, whose proof is mainly based on the contraction mapping principle and improved Strichartz estimates.

\begin{thm}\label{pb wellposedness}
Let $N\geq 2$, $\frac 1 2 <s_2<s_1< 1$ and $2< p<\frac{2N}{N-2s_1}$. Then, for any $\psi_0\in  H_{rad}^{s_1}(\R^N)$, there exist a constant $T:=T(\|\psi_0\|_{H^{s_1}})>0$ and a unique maximal solution $\psi \in C([0, T), H_{rad}^{s_1}(\R^N))$ to \eqref{evolv pb0} satisfiing the alternative: either $T=+\infty$ or $T<+\infty$ and
$$
\lim_{t \to T^{-}}  \|(-\Delta)^{\frac{s_1}{2}} \psi\|_2 =+\infty.
$$
In addition, there holds that
\begin{enumerate}
\item [$(\textnormal{i})$] $\psi \in L_{loc}^{\frac{4s_1p}{N(p-2)}}([0,T),W^{s_1,p}(\R^N))$.
\item  [$(\textnormal{ii})$]The solution $\psi(t)$ satisfies the conservation of the mass and the energy, i.e. $\|\psi(t)\|_2=\|\psi_0\|_2$ and $E(\psi(t))=E(\psi_0)$ for any $t \in [0, T)$.
\item  [$(\textnormal{iii})$]The solution $\psi(t)$ exists globally in time if $p<2+\frac{4s_1}{N}$ or $p=2+\frac{4s_1}{N}$ and
$$
 \|\psi_0\|_2<\left(\frac{N+2s_1}{NC_{N,s_1}}\right)^{\frac{N}{4s_1}},
$$
where $C_{N, s_1}=C_{N,p,s_1}>0$ is the optimal constant appearing in \eqref{gn} for $p=2+\frac{4s_1}{N}$.
\end{enumerate}
\end{thm}

For further clarifications, we need to introduce a function $\phi \in H^{s_1}(\R^N)$ as the ground state solution to the following fractional nonlinear elliptic equation,
\begin{align}\label{psi eqt}
(-\Delta)^{s_1}\phi+\phi=\phi^{p-1}.
\end{align}
In fact, it turns out in \cite{Rup13,Rup16} that $\phi$ is positive, radially symmetric and decreasing. Moreover, whenever $ 2+\frac{4s_1}{N}\leq p <\frac{2N}{N-2s_1}$ and $N \geq 2$, we define
$$
0 \leq s_{c}:=\frac{N}{2}-\frac{2s_1}{p-2} <s_1, \quad \sigma_c:=\frac{s_1-s_c}{s_c}>0.
$$
It should be noted that $s_c>0$ if $p>2+\frac{4s_1}{N}$ and $s_c=0$ if $p=2+\frac{4s_1}{N}$. We also define a functional by
$$
\mathcal{E}(\phi):=\frac 12 \int_{\R^N} |(-\Delta)^{\frac{s_1}{2}} \phi |^2 \,dx-\frac 1 p \int_{\R^N} |\phi|^p \,dx.
$$

\begin{thm}\label{blow-up vs global solt}
Let $N\geq 2$, $\frac1 2<s_2<s_1<1$ and $p \geq 2 +\frac{4s_1}{N}$. Let $\psi \in C([0, T), H^{s_1}_{rad}(\R^N))$ be the solution to \eqref{evolv pb0} with initial datum $\psi_0 \in H^{s_1}_{rad}(\R^N)$ and $\phi \in H^{s_1}_{rad}(\R^N)$ be the ground state solution to \eqref{psi eqt}.
\begin{enumerate}
\item [$(\textnormal{i})$] If $s_c>0$ and $\psi_0 \in H^{s_1}_{rad}(\R^N)$ satisfies
\begin{align}\label{energ u0 inf energ gs}
E(\psi_0)M(\psi_0)^{\sigma_c} <\mathcal{E}(\phi)M(\phi)^{\sigma_c},
\end{align}
\begin{align}\label{mass u0 sup mass gs}
\|(-\Delta)^{\frac{s_1}{2}} \psi_0\|_2\|\psi_0\|_2^{\sigma_c} < \|(-\Delta)^{\frac{s_1}{2}} \phi\|_2\|\phi\|_2^{\sigma_c},
\end{align}
then $\psi(t)$ exists globally in time, i.e. $T=+\infty$.
\item [$(\textnormal{ii})$] If $s_c>0$ and $2<p < 2+4s_1$, either $E(\psi_0)<0$ or $E(\psi_0)\geq 0$ satisifes  \eqref{energ u0 inf energ gs} and the following condition,
\begin{align}\label{mass u0 sup mass gs1}
\|(-\Delta)^{\frac{s_1}{2}} \psi_0\|_2\|\psi_0\|_2^{\sigma_c} > \|(-\Delta)^{\frac{s_1}{2}} \phi\|_2\|\phi\|_2^{\sigma_c},
\end{align}
then $\psi(t)$ blows up in finite time and
$$
\limsup_{t\rightarrow T^-} \|(-\Delta)^{\frac{s_1}{2}} \psi\|_2=+\infty.
$$
\item [$(\textnormal{iii})$] If $s_c=0$ and $E(\psi_0)<0$, then $\psi(t)$ either blows up in finite time or blows up in infinite time satisfying there exist $C>0$ and $t^\ast>0$ such that
$$
\|(-\Delta)^{\frac{s_1}{2}} \psi\|_2 + \|(-\Delta)^{\frac{s_2}{2}} \psi\|_2\geq Ct^{s_1}, \quad \forall \,\, t \geq t^*.
$$
\end{enumerate}
\end{thm}

The proof of Theorem \ref{blow-up vs global solt} crucially relies on the variational characteristics of $\phi$ and the analysis of the evolution of the following localized virial type quantity,
$$
M_{\chi_R}[\psi(t)]:=2 Im \int_{\R^N}\overline{\psi}\nabla\chi_R \cdot \nabla \psi\, dx,
$$
where $\chi_R: \R^N \to \R^+$ is a proper cut-off function.

Finally we are going to address orbital instability of ground state solutions to \eqref{fequ}-\eqref{mass} in the following sense.

\begin{defi}
We say that a solution $u \in H^{s_1}(\R^N)$ to \eqref{fequ} is orbitally unstable, if for any $\eps > 0$ there exists $ v \in H^{s_1}(\R^N)$ such that $\|v-u\|_{H^{s_1}} \leq \eps$ and the solution $\psi (t)$ to \eqref{evolv pb0} with initial datum $\psi(0) = v$ blows up in finite or infinite time.
\end{defi}

\begin{thm}\label{thm5}
Let $N \geq 2$, $\frac 1 2 <s_2<s_1<1$ and $p \geq 2 + \frac{4s_1}{N}$. Then standing waves associated with ground state solutions to \eqref{fequ}-\eqref{mass} are orbitally unstable by blowup in finite or infinite time. In addition, if $2<p<2+4s_1$, then they are orbitally unstable by blowup in finite time.
\end{thm}

{\noindent \bf Structure of the Paper.} The paper is organized as follows. In Section \ref{section2}, we present some preliminary results and give the proof of Theorem \ref{thm1}. In Section \ref{section3}, we consider the existence of ground state solutions to \eqref{fequ}-\eqref{mass} and show the proofs of Theorems \ref{thm2} and \ref{thm6}. In Section \ref{section4}, we aim to prove the existence of bound state solutions to \eqref{fequ}-\eqref{mass} and give the proof of Theorems \ref{thm3}. In Section \ref{section5}, we discuss some properties of the function $c \mapsto \gamma(c)$ and present the proof of Theorem \ref{thm4}. Section \ref{section6} is devoted to the study of the local well-posednesss of solutions to \eqref{evolv pb0} and contains the proof of Theorem \ref{pb wellposedness}.  Section \ref{section7} is devoted the proof of Theorem \ref{blow-up vs global solt}. In Section \ref{section8}, orbital instability of ground state solutions to \eqref{fequ}-\eqref{mass} is discussed, i.e. Theorem \ref{thm5} is established. In Appendix, we deduce the Pohozaev identity satisfied by solutions to \eqref{fequ}.

\begin{notation}
Throughout the paper, $L^r(\R^N)$ denotes the usual Lebesgue space equipped with the norm
$$
\|f\|_r:=\left(\int_{\R^N}|f(x)|^r \, dx \right)^{\frac{1}{r}}, \quad 1 \leq r<\infty, \quad  \|f\|_\infty:= \underset{x\in \R^N}{\mbox{ess sup}} \, |f(x)|.
$$
Moreover, $W^{s,r}(\R^N)$ denotes the usual Sobolev space equipped with the norm
$$
\|f\|_{W^{s,r}}:=\|f\|_2+ \|(-\Delta)^{\frac{s}{2}}f\|_2, \quad 0<s<1.
$$
In the case $r=2$, we use $H^s(\R^N)$ to denote $W^{s,2}(\R^N)$ and use $H^{s}_{rad}(\R^N)$ to denote the subspace of $H^{s}(\R^N)$, consisting of radially symmetric functions in $H^{s}(\R^N)$. The real number $r^{\prime}:=\frac{r}{r-1}$ is the conjugate exponent associate to a number $r \geq 1$ with the convention $1^\prime=\infty$ and $\infty^\prime=1$.
We also need to introduce some B\"{o}chner spaces which are denoted by $L_T^qL_x^r:=L^q([0,T),L^r(\R^N))$ equipped with the natural norms. If $X$ is an abstract space, then the set of continuous functions defined on $[0,T)$ and valued in $X$ is denoted by $C_T(X):=C([0,T),X)$, if necessary the interval of time may be closed.
If $A$ and $B$ are two nonnegative quantities, we write $A\lesssim B$ to denote $A\leq CB$. We write $A\sim B$ if $A\lesssim B$ and $B\lesssim A$ hold.
Whenever $\varepsilon_n\rightarrow0$ as $n$ goes to infinity, we denote $o_n(1)=\varepsilon_n$. 
\end{notation}

\section{Preliminaries and proof of Theorem \ref{thm1}} \label{section2}

In this section, we shall present some preliminary results used to prove our main theorems and give the proof of Theorem \ref{thm1}. First of all, let us display the well-known Gagliardo-Nirenberg inequality in $H^{s_1}(\R^N)$.

\begin{lem} \label{gninequ}
Let $0<s<1$, $2 \leq p< \infty$ if $N<2s$ and $2 \leq p<\frac{2N}{N-2s}$ if $N>2s$, then
\begin{align} \label{gn}
\int_{\R^N} |u|^p\,dx \leq C_{N,p,s} \left(\int_{\R^N} |(-\Delta)^{\frac s 2} u|^2 \,dx \right)^{\frac{N(p-2)}{4s}}\left(\int_{\R^N}|u|^2 \,dx \right)^{\frac p2 -\frac{N(p-2)}{4s}},
\end{align}
where $C_{N,p,s}>0$ denotes the optimal constant.
\end{lem}

\begin{lem} \label{pohozaev}
Let $u \in H^{s_1}(\R^N)$ is a solution to \eqref{fequ}-\eqref{mass}, then $Q(u)=0$, i.e.
$$
s_1\int_{\R^N} |(-\Delta)^{\frac{s_1}{2}} u|^2\,dx + s_2 \int_{\R^N} |(-\Delta)^{\frac{s_2}{2}} u|^2\,dx=\frac{N(p-2)}{2p} \int_{\R^N}|u|^p\,dx.
$$
\end{lem}
\begin{proof}
For convenience of readers, the proof of this lemma shalled be postponed to Appendix.
\end{proof}

Making use of Lemmas \ref{gninequ} and \ref{pohozaev}, we are now able to prove Theorem \ref{thm1}.

\begin{proof}[Proof of Theorem \ref{thm1}] In view of \eqref{gn}, we first have that, for any $u \in S(c)$,
$$
E(u) \geq \frac 12 \left(1-\left(\frac{c}{c_{N,s_1}}\right)^{\frac{2s_1}{N}}\right) \int_{\R^N} |(-\Delta)^{\frac{s_1}{2}} u|^2 \, dx + \frac 12  \int_{\R^N} |(-\Delta)^{\frac{s_2}{2}} u |^2 \, dx.
$$
This clearly shows that $m(c)\geq 0$ for any $0<c \leq c_{N, s_1}$. On the other hand, from \eqref{scaling}, we can deduce that $E(u_t) \to 0$ as $t \to 0^+$. This leads to $m(c) \leq 0$ for any $c>0$. Therefore, we obtain that $m(c)=0$ for any $0<c \leq c_{N, s_1}$. We next prove that $m(c)$ cannot be attained for any $0<c<c_{N,s_1}$. Let us suppose that $m(c)$ is attained for some $0<c\leq c_{N,s_1}$. Hence there exists $u \in S(c)$ such that $m(c)=E(u)$. From Lemma \ref{pohozaev}, we get that $Q(u)=0$.
Using \eqref{gn}, we then see that
\begin{align}\label{criticalineq}
\begin{split}
s_1\int_{\R^N} |(-\Delta)^{\frac{s_1}{2}} u|^2\, dx +s_2\int_{\R^N} |(-\Delta)^{\frac{s_2}{2}} u|^2\, dx
&=\frac{Ns_1}{N+2s_1} \int_{\R^N} |u|^{2+\frac{4s_1}{N}} \,dx \\
& \leq s_1\left(\frac{c}{c_{N, s_1}}\right)^{\frac{2s_1}{N}}\int_{\R^N} |(-\Delta)^{\frac{s_1}{2}} u|^2\,dx.
\end{split}
\end{align}
This then suggests that $u=0$, because of $0<c\leq c_{N,s_1}$. As a result, we derive that $m(c)$ is not attained for any $0<c \leq c_{N,s_1}$. From the discussions above, we can also conclude that \eqref{fequ}-\eqref{mass} has no solutions for any $0<c \leq  c_{N, s_1}$. We now prove that $m(c)=-\infty$ for any $c>c_{N, s_1}$. Let $ u \in H^{s_1}(\R^N)$ be such that the optimal constant $C_{N, s_1}$ in \eqref{gn} is achieved for $p=2+\frac{4s_1}{N}$. This means that
\begin{align} \label{optimal}
\int_{\R^N}|u|^{2+\frac{4s_1}{N}} \,dx =C_{N,s_1} \left(\int_{\R^N}|(-\Delta)^{\frac {s_1} {2}}u|^2 \,dx\right)\left( \int_{\R^N}|u|^2 \, dx \right)^{\frac{2s_1}{N}}.
\end{align}
Define
\begin{align}  \label{defw}
w:=c^{\frac 12}\frac{u}{\|u\|_2} \in S(c).
\end{align}
By applying \eqref{optimal}, we can derive that
\begin{align} \label{scaling1}
\begin{split}
E(w_t)&=\frac{c}{2\|u\|_2^2} t^{2s_1}\int_{\R^N} |(-\Delta)^{\frac{s_1}{2}} u|^2 \, dx +\frac{c}{2\|u\|_2^2} t^{2s_2}\int_{\R^N} |(-\Delta)^{\frac{s_2}{2}} u|^2 \, dx \\
&\quad -\frac{N}{2N+4s_1}\frac{c^{1+\frac{2s_1}{N}}}{\|u\|_2^{2+\frac{4s_1}{N}}} t^{2s_1} \int_{\R^N}|u|^{2+\frac{4s_1}{N}}  \,dx\\
&=\frac{c}{2\|u\|_2^2} \left(1-\left(\frac{c}{c_{N,s_1}}\right)^{\frac{2s_1}{N}}\right) t^{2s_1}\int_{\R^N} |(-\Delta)^{\frac{s_1}{2}}u|^2 \, dx +\frac{c}{2\|u\|_2^2} t^{2s_2} \int_{\R^N} |(-\Delta)^{\frac{s_2}{2}}u|^2\, dx.
\end{split}
\end{align}
This indicates that $E(w_t) \to -\infty$ as $t \to \infty$ for any $c>c_{N, s_1}$, due to $0<s_2<s_1<1$. Hence $m(c)=-\infty$ for any $c>c_{N, s_1}$. This completes the proof.
\end{proof}

\begin{lem}\label{pohoz 0}
Let $\phi \in H^{s_1}(\R^N)$ be the ground state solution to \eqref{psi eqt}, then
\begin{align}\label{pohoz 1}
\|(-\Delta)^{\frac {s_1} {2}} \phi\|_2^2=\frac{N(p-2)}{2s_1p-N(p-2)}\|\phi\|_2^2,
\end{align}
\begin{align}\label{pohoz 2}
\|\phi\|_p^p=\frac{2s_1p}{N(p-2)}\|(-\Delta)^{\frac {s_1} {2}} \phi\|_2^2,
\end{align}
\begin{align}\label{best const with GS}
\displaystyle C_{N,p,s_1}=\frac{2s_1p}{N(p-2)}\left(\frac{N(p-2)}{2s_1p-N(p-2)}\right)^{\frac{4s_1-Np+2N}{4s_1}}\|\phi\|_2^{2-p},
\end{align}
where $C_{N,p,s_1}>0$ is the optimal constant in \eqref{gn} with $s=s_1$.
\end{lem}
\begin{proof}
Multiplying \eqref{psi eqt} against $\phi$ and integrating on $\R^N$, one gets that
\begin{align}\label{eq 1}
\|(-\Delta)^{\frac{s_1}{2}}\phi\|_2^2+\|\phi\|_2^2=\|\phi\|_p^p,
\end{align}
On the other hand, from \cite[Lemma 8.1]{Rup16}, one has that
\begin{align}\label{eq 2}
\frac{N-2s_1}{2}\|(-\Delta)^{\frac{s_1}{2}}\phi\|_2^2+\frac{N}{2}\|\phi\|_2^2=\frac{N}{p}\|\phi\|_p^p.
\end{align}
It then follows from \eqref{eq 1} and \eqref{eq 2} that
\begin{align*}
\|(-\Delta)^{\frac{s_1}{2}}\phi\|_2^2=\frac{N(p-2)}{2s_1p-N(p-2)}\|\phi\|_2^2.
\end{align*}
This along with \eqref{eq 1} then yields to
\begin{align*}
\|\phi\|_p^p
& = \frac{2s_1p}{N(p-2)}\|(-\Delta)^{\frac{s_1}{2}}\phi\|_2^2.
\end{align*}
From straightforward calculations and the equality
$$
\|\phi\|_{p}^p =C_{N,p,s_1} \|(-\Delta)^{\frac s 2} \phi\|_2^{\frac{N(p-2)}{2s_1}}\|\phi\|_2^{p-\frac{N(p-2)}{2s}},
$$
then \eqref{best const with GS} follows. Thus the proof is completed.
\end{proof}

\section{Existence and characteristics of ground state solutions} \label{section3}

In this section, we shall present the proofs of Theorems \ref{thm2} and \ref{thm6}. For this aim, we first need to establish some preliminary results. Let us remind that $c_0=c_{N,s_1}$ if $p=2 +\frac{4s_1}{N}$ and $c_0=0$ if $p> 2 +\frac{4s_1}{N}$, where $c_{N,s_1}>0$ is the constant given in Theorem \ref{thm1}.

\begin{lem} \label{nonempty}
Let $N \geq 1$, $0<s_2<s_1<1$ and $p \geq 2 +\frac{4s_1}{N}$, then $P(c) \neq \emptyset$ for any $c>c_0$.
\end{lem}
\begin{proof}
For the case $p=2+\frac{4s_1}{N}$ and $c>c_{N,s_1}$, it follows from \eqref{scaling1} that $E(w_t)>0$ for any $t>0$ small enough and $E(w_t)<0$ for any $t>0$ large enough, where $w \in S(c)$ is defined by \eqref{defw}. Therefore, one finds that there exists a constant $t_w>0$ such that
$$
t_wQ(w_{t_w})=\frac{d E(w_t)}{d t}{\mid_{t=t_w}}=0.
$$
Hence we have that $w_{t_w} \in P(c)$ and $P(c) \neq \emptyset$ for any $c>c_{N, s_1}$. For the case $p>2+\frac{4s_1}{N}$ and $c>0$, from applying \eqref{scaling}, we can obtain the desired conclusion by a similar way. Thus the proof is completed.
\end{proof}

\begin{lem} \label{coercive}
Let $N \geq 1$, $0<s_2<s_1<1$ and $p \geq 2 +\frac{4s_1}{N}$, then, for any $c>c_0$, $E$ restricted on $P(c)$ is coercive and bounded from below by a positive constant. 
\end{lem}
\begin{proof}
If $u \in P(c)$, then $Q(u)=0$. As a result, there holds that
\begin{align} \label{ide1}
\begin{split}
\hspace{-1cm}E(u)&=E(u)-\frac{2}{N(p-2)}Q(u)\\
&=\frac{N(p-2)-4s_1}{2N(p-2)}\|(-\Delta)^{\frac{s_1}{2}} u\|_2^2 +\frac{N(p-2)-4s_2}{2N(p-2)}\|(-\Delta)^{\frac{s_2}{2}} u\|_2^2.
\end{split}
\end{align}
We first consider the case $p=2+\frac{4s_1}{N}$. In this case, we shall prove that $E$ restricted on $P(c)$ is coercive for any $c>c_{N,s_1}$. To do this, we argue by contradiction that there exists a sequence $\{u_n\} \subset P(c)$ satisfying $\|u_n\|_{H^{s_1}} \to \infty$ as $n \to \infty$ such that $\{E(u_n)\} \subset \R$ is bounded for some $0<c<c_{N,s_1}$. From \eqref{ide1}, we then infer that $\{\|(-\Delta)^{{s_2}/{2}} u_n\|_2\} \subset \R$ is bounded. Next we are going to deduce that $\{\|(-\Delta)^{{s_1}/{2}} u_n\|_2\} \subset \R$ is bounded. If $N=1$ and $2s_2\geq 1$ or $N \geq 1$, $2s_2<N$ and $2 < p \leq \frac{2N}{N-2s_2}$, it then follows from \eqref{gn} in that $\{\|u_n\|_p\} \subset \R$ is bounded, because of $\{\|(-\Delta)^{{s_2}/{2}} u_n\|_2\} \subset \R$ is bounded. Thanks to $Q(u_n)=0$, we then get that $\{\|(-\Delta)^{{s_1}/{2}} u_n\|_2\} \subset \R$ is bounded. If $N \geq 1$, $2s_2<N$ and $p >\frac{2N}{N-2s_2}$, then there exist $p_1, p_2 >0$ with $2<p_2<\frac{2N}{N-2s_2}<p_1$ and $0<\theta<1$ such that $p=\theta p_1 +(1-\theta)p_2$. Using H\"older's inequality, we then derive that
\begin{align} \label{inter}
\|u_n\|_p^p \leq \|u_n\|_{p_1}^{\theta p_1} \|u_n\|_{p_2}^{(1-\theta)p_2}.
\end{align}
On the other hand, applying \eqref{gn}, one gets that
$$
\|u_n\|_{p_1}^{p_1} \leq C_{N,p_1,s_1} \|(-\Delta)^{\frac {s_1} {2}} u_n\|_2^{\frac{N(p_1-2)}{2s_1}}\|u_n\|_2^{p_1-\frac{N(p_1-2)}{2s_1}}
$$
and
$$
\|u_n\|_{p_2}^{p_2} \leq C_{N,p_2,s_2}  \|(-\Delta)^{\frac {s_2} {2}} u_n\|_2^{\frac{N(p_2-2)}{2s_2}}\|u_n\|_2^{p_2-\frac{N(p_2-2)}{2s_2}}.
$$
This along with \eqref{inter} results in
\begin{align} \label{inequ1}
\|u_n\|_p^p \leq C\|(-\Delta)^{\frac {s_1} {2}} u_n\|_2^{\frac{\theta N(p_1-2)}{2s_1}} \|(-\Delta)^{\frac {s_2} {2}} u_n\|_2^{\frac{(1-\theta )N(p_2-2)}{2s_2}} \|u_n\|_2^{p-\frac{\theta N(p_1-2)}{2s_1}-\frac{(1-\theta)N(p_2-2)}{2s_2}},
\end{align}
where $ C:=C_{N,p_1,s_1}^{\theta}C_{N,p_2,s_2}^{1-\theta}>0$. Observe that
$$
\theta(p_1-2)=(p-2)-(1-\theta)(p_2-2)<p-2=\frac{4s_1}{N}.
$$
Therefore, there holds that
$$
\frac{\theta N(p_1-2)}{2s_1}<2.
$$
Since $Q(u_n)=0$ and $\{\|(-\Delta)^{{s_2}/{2}} u_n\|_2\} \subset \R$ is bounded, from \eqref{inequ1}, then $\{\|(-\Delta)^{{s_1}/{2}} u_n\|_2\} \subset \R$ is bounded. Hence we reach a contradiction, because $\|u_n\|_{H^{s_1}} \to \infty$ as $n \to \infty$. This in turn implies that $E$ restricted on $P(c)$ is coercive. We now demonstrate that $E$ restricted on $P(c)$ is bounded from below by a positive constant for any $c>c_{N,s_1}$.  If $N=1$ and $2s_2\geq 1$ or $N \geq 1$, $2s_2<N$ and $2<p \leq \frac{2N}{N-2s_2}$, it then yields from \eqref{gn} that, for any $u \in P(c)$,
\begin{align} \label{ide11}
\begin{split}
&s_1\int_{\R^N} |(-\Delta)^{\frac{s_1}{2}} u|^2\, dx +s_2\int_{\R^N} |(-\Delta)^{\frac{s_2}{2}} u|^2\, dx=\frac{Ns_1}{N+2s_1} \int_{\R^N}|u|^p\,dx \\
&\leq \frac{Ns_1 C_{N,p,s_2}}{N+2s_1} \left(\int_{\R^N} |(-\Delta)^{\frac {s_2} 2} u|^2 \,dx \right)^{\frac{N(p-2)}{4s_2}}\left(\int_{\R^N}|u|^2 \,dx \right)^{\frac p2 -\frac{N(p-2)}{4s_2}}.
\end{split}
\end{align}
Note that
$$
\frac{N(p-2)}{4s_2}>1.
$$
Therefore, by \eqref{ide11}, we get that $\|(-\Delta)^{{s_2}/{2}} u\|_2$ is bounded from below by a positive constant. Coming back to \eqref{ide1}, we then have the desired result. If $N \geq 1$, $2s_2<N$ and $p >\frac{2N}{N-2s_2}$, it then follows from \eqref{criticalineq} that $\|(-\Delta)^{{s_1}/{2}} u\|_2$ is bounded from below by a positive constant, because of $c>c_{N,s_1}$. Let us claim that $\|(-\Delta)^{{s_2}/{2}} u\|_2$ is also bounded from below by a positive constant. Otherwise, we may suppose that there exists $\{u_n\}\subset P(c)$ such that $\|(-\Delta)^{{s_2}/{2}} u_n\|_2=o_n(1)$. This leads to $\|u_n\|_{\frac{2N}{N-2s_1}}=o_n(1)$. If $\{(-\Delta)^{{s_1}/{2}} u_n\|_2\} \subset \R$ is bounded, we then derive from H\"older's inequality that $\|u_n\|_p=o_n(1)$. Since $Q(u_n)=0$, it then gives that $\|(-\Delta)^{{s_1}/{2}} u_n\|_2=o_n(1)$, which is a contradiction. If $\{\|(-\Delta)^{{s_1}/{2}} u_n\|_2\} \subset \R$ is unbounded, we get by the coerciveness of $E$ restricted on $P(c)$ that $E(u_n) \to \infty$ as $n \to \infty$. In view of \eqref{ide1}, we then obtain that $\|(-\Delta)^{{s_2}/{2}} u_n\|_2 \to \infty$ as $n \to \infty$. This is impossible, because we assumed that $\|(-\Delta)^{{s_2}/{2}} u_n\|_2=o_n(1)$. Therefore, the claim holds true. Utilizing \eqref{ide1}, we get the desired result.

Next we consider the case $p>2+\frac{4s_1}{N}$. In this case, there holds that
$$
\frac{N(p-2)}{4s_1}>1.
$$
It is immediate to find from \eqref{ide1} that $E$ restricted on $P(c)$ is coercive for any $c>0$. We now prove that $E$ restricted on $P(c)$ is bounded from below by a positive constant. For any $u \in P(c)$, we know that $Q(u)=0$.  It then follows from \eqref{gn} that
\begin{align*}
\hspace{-1cm}s_1\int_{\R^N} |(-\Delta)^{\frac{s_1}{2}} u|^2\, dx +s_2\int_{\R^N} |(-\Delta)^{\frac{s_2}{2}} u|^2\, dx&=\frac{Ns_1}{N+2s_1} \int_{\R^N}|u|^p \,dx \\
&\leq C_{N,s_1}c^{\frac p 2 -\frac{N(p-2)}{4s_1}}\left(\int_{\R^N} |(-\Delta)^{\frac{s_1}{2}} u|^2dx\right)^{\frac{N(p-2)}{4s_1}}.
\end{align*}
This implies that $\|(-\Delta)^{{s_1}/{2}} u\|_2$ is bounded from below by a positive constant. The proof is completed by applying \eqref{ide1}.
\end{proof}

\begin{lem} \label{monotonicity}
Let $N \geq 1$, $0<s_2<s_1<1$ and $p \geq 2 +\frac{4s_1}{N}$. Let $u \in S(c)$ for $c >c_0$. Assume in addition that $ \sup_{t \geq 0} E(u_t)<\infty$ if $p=2+\frac{4s_1}{N}$. Then there exists a unique $t_u>0$ such that $u_{t_u} \in P(c)$ and $E(u_{t_u})=\sup_{t \geq 0} E(u_t)$. Moreover, the function $t \mapsto E(u_t)$ is concave on $[t_u,\infty)$ and $0<t_u < 1$ if $Q(u) < 0$.
\end{lem}
\begin{proof}
For any $u \in S(c)$, by using \eqref{scaling}, we first have that
\begin{align}\label{derivative} \nonumber
\frac{d}{d t}E(u_t)&=s_1 {t^{2s_1-1}}\int_{\R^N} |(-\Delta)^{\frac{s_1}{2}} u|^2 \, dx + s_2{t^{2s_2-1}} \int_{\R^N} |(-\Delta)^{\frac{s_2}{2}} u|^2 \, dx -\frac{N(p-2)}{2p}{t^{N(\frac p 2 -1)-1}} \int_{\R^N} |u|^p dx \\
&=\frac 1 t Q(u_t).
\end{align}
If $p=2+\frac{4s_1}{N}$ and $\displaystyle\sup_{t \geq 0} E(u_t)<\infty$, it then follows from \eqref{scaling} that
$$
\frac 12 \int_{\R^N} |(-\Delta)^{\frac{s_1}{2}} u|^2 \, dx <\frac 1 p \int_{\R^N} |u|^p \, dx.
$$
In light of \eqref{derivative}, we then easily derive there exists a unique $t_u>0$ such that $Q(u_{t_u})=0$. Notice that $Q(u_t)>0$ for any $0<t<t_u$ and $Q(u_t)<0$ for any $t>t_u$, then $E(u_{t_u})=\sup_{t \geq 0} E(u_t)$. Furthermore, if $Q(u) < 0$,  we then obtain that $0<t_u < 1$. If $p>2+\frac{4s_1}{N}$, then
$$
\frac {N (p-2)}{2}>2s_1.
$$
In this case, by means of \eqref{derivative}, we can also get the desired result. We now show that the function $t \mapsto E(u_t)$ is concave on $[t_u, \infty)$. Observe that
\begin{align*}
\frac{d^2}{d t^2}E(u_t)&=s_1(2s_1-1) {t^{2s_1-2}}\int_{\R^N} |(-\Delta)^{\frac{s_1}{2}} u|^2 \, dx + s_2(2s_2-1){t^{2s_2-2}} \int_{\R^N} |(-\Delta)^{\frac{s_2}{2}} u|^2 dx \\
&\quad -\frac{N(p-2)(N(p-2)-2)}{4p}{t^{N(\frac p 2 -1)-2}} \int_{\R^N} |u|^p dx.
\end{align*}
If $2s_1 =1$, we then have that $\frac{d^2}{d t^2}E(u_t)<0$ for any $t >0$. We next consider the case $2s_1 \neq 1$. Let us write $t=\gamma t_u$ for $\gamma >1$, then
\begin{align*}
\frac{d^2}{d t^2}E(u_t)&=\frac{2s_1-1}{\gamma^{2-2s_1} t^2_u} \left(s_1{t^{2s_1}_u}\int_{\R^N} |(-\Delta)^{\frac{s_1}{2}} u|^2 dx + \frac{s_2(2s_2-1)}{2s_1-1}\gamma^{2(s_2-s_1)}{t^{2s_2}_u} \int_{\R^N} |(-\Delta)^{\frac{s_2}{2}} u|^2  dx \right. \\
& \quad \left.-\frac{N(p-2)(N(p-2)-2)}{4p(2s_1-1)}{t^{N(\frac p 2 -1)}_u} \gamma^{{N(\frac p 2 -1)}-2s_1} \int_{\R^N} |u|^pdx \right)\\
&=:\frac{2s_1-1}{\gamma^{2-2s_1} t^2_u} g(\gamma).
\end{align*}
Since $Q(u_{t_u})=0$, i.e.
$$
s_1{t^{2s_1}_u}\int_{\R^N} |(-\Delta)^{\frac{s_1}{2}} u|^2  dx + s_2{t^{2s_2}_u} \int_{\R^N} |(-\Delta)^{\frac{s_2}{2}} u|^2  dx=\frac{N(p-2)}{2p}{t^{N(\frac p 2 -1)}_u} \int_{\R^N} |u|^p dx,
$$
then
\begin{align*}
g(\gamma)&=s_2\left(\frac{2s_2-1}{2s_1-1}\gamma^{2(s_2-s_1)}-1\right){t^{2s_2}_u} \int_{\R^N} |(-\Delta)^{\frac{s_2}{2}} u|^2 \, dx \\
&\quad  +\frac{N(p-2)}{2p}\left(1-\frac{(N(p-2)-2)}{2(2s_1-1)}\gamma^{{N(\frac p 2 -1)}-2s_1} \right){t^{N(\frac p 2 -1)}_u} \int_{\R^N} |u|^p \,dx.
\end{align*}
If $2s_1>1$, we then deduce that $g(\gamma)<0$ for any $\gamma >1$. If $0<2s_1<1$, we then find that $g(\gamma) \to -\infty$ as $\gamma \to 0^+$, $g(1)>0$ and $g(\gamma) \to \infty$ as $\gamma \to \infty$. This together with the monotonicity of $g$ shows that $g(\gamma)>0$ for any $\gamma >1$. Consequently, we obtain that $\frac{d^2}{d t^2}E(u_t)<0$ for any $t \geq t_u$. This finishes the proof.
\end{proof}

\begin{defi}\label{homotopy} \cite[Definition 3.1]{Gh}
Let $B$ be a closed subset of a set $Y \subset H^{s_1}(\R^2)$. We say that a class $\mathcal{G}$ of compact subsets of $Y$ is a homotopy stable family with the closed boundary $B$ provided that
\begin{enumerate}
\item [\textnormal{(i)}] every set in $\mathcal{G}$ contains $B$;
\item [\textnormal{(ii)}] for any $A \in \mathcal{G}$ and any function $\eta \in C([0, 1] \times Y, Y)$ satisfying $\eta(t, x)=x$ for all $(t, x) \in (\{0\} \times Y) \cup([0, 1] \times B)$, then $\eta(\{1\} \times A) \in \mathcal{G}$.
\end{enumerate}
\end{defi}

Let us remark that $B=\emptyset$ is admissible. For further discussions, we shall introduce some notations. If $p=2+\frac{4s_1}{N}$, we define a functional $F: \mathcal{S}(c) \to R$ by $F(u):=E(u_{t_u})=\max_{t>0}E(u_t)$, where
\begin{align} \label{sc}
\mathcal{S}(c):=\left\{u \in S(c) : \frac 12 \int_{\R^N} |(-\Delta)^{\frac{s_1}{2}} u|^2 \, dx <\frac 1 p \int_{\R^N} |u|^p  dx\right\}.
\end{align}
If $p>2+\frac{4s_1}{N}$, we define a functional $F: S(c) \to R$ by $F(u):=E(u_{t_u})=\displaystyle\max_{t>0}E(u_t)$.

\begin{lem}\label{ps}
Let $N \geq 1$, $0<s_2<s_1<1$ and $p \geq 2 +\frac{4s_1}{N}$. Let $\mathcal{G}$ be a homotopy stable family of compact subsets of $S(c)$ with closed boundary $B$ and set
\begin{align} \label{ming}
\gamma_{\mathcal{G}}(c):=\inf_{A\in \mathcal{G}}\max_{u \in A} F(u).
\end{align}
Suppose that $B$ is contained on a connected component of $P(c)$ and $\max\{\sup F(B), 0\}<\gamma_{\mathcal{G}}(c)<\infty$. Then there exists a Palais-Smale sequence $\{u_n\} \subset P(c)$ for $E$ restricted on $S(c)$ at the level $\gamma_{\mathcal{G}}(c)$ for any $c>c_0$.
\end{lem}
\begin{proof}
The proof benefits from ingredients developed in \cite{BS1, BS2}. For simplicity, we only show the proof for the case $p>2+ \frac{4s_1}{N}$. Replacing the role of $S(c)$ by $\mathcal{S}(c)$, one can similarly treat the case $p=2+ \frac{4s_1}{N}$. To begin with, we define a mapping $\eta: [0,1] \times S(c) \to S(c)$ by $\eta(s, u)=u_{1-s+st_u}$. From Lemma \ref{monotonicity}, we have that $t_u=1$ if $u \in P(c)$. Thus we see that $\eta(s ,u)=u$ for any $(s, u) \in (\{0\} \times S(c)) \cup([0, 1] \times B)$, because of $B \subset P(c)$. In addition, it is simple to derive that $\eta$ is continuous on $[0,1] \times S(c)$. Let $\{D_n\} \subset \mathcal{G}$ be a minimizing sequence to \eqref{ming}. By Definition \ref{homotopy}, we then get that
$$
A_n:=\eta(\{1\} \times D_n)=\{u_{t_u} : u \in D_n\} \in \mathcal{G}.
$$
Note that $A_n \subset P(c)$, it then holds that
$$
\displaystyle\max_{v \in A_n}F(v)=\displaystyle\max_{u \in D_n}F(u).
$$
Therefore, there exists another minimizing sequence $\{A_n\} \subset P(c)$ to \eqref{ming}. Using \cite[Theorem 3.2]{Gh}, we then deduce that there exists a Palais-Smale sequence $\{\tilde{u}_n\} \subset S(c)$ for $F$ at the level $\gamma_{\mathcal{G}}(c)$ such that $\mbox{dist}_{H^{s_1}}(\tilde{u}_n, A_n)=o_n(1)$. For simplicity, we shall write $t_n=t_{\tilde{u}_n}$ and $u_n=(\tilde{u}_n)_{t_n}$.

We now claim that there exists a constant $C>0$ such that $1/C \leq t_n \leq C$. Indeed, notice first that
$$
t_n^{2s_1}=\frac{\int_{\R^N} |(-\Delta)^{\frac{s_1}{2}} u_n|^2 \, dx}{\int_{\R^N} |(-\Delta)^{\frac{s_1}{2}} \tilde{u}_n|^2 \, dx}.
$$
Since $E(u_n)=F(\tilde{u}_n)=\gamma_{\mathcal{G}}(c)+o_n(1)$ and $\{u_n\} \subset P(c)$, it then follows from Lemma \ref{coercive} that there exists a constant $C_1>0$ such that $1/C_1 \leq \|u_n\|_{H^{s_1}} \leq C_1$. On the other hand, since $\{A_n\} \subset P(c)$ is a minimizing sequence to \eqref{ming}, by Lemma \ref{coercive}, we then have that $\{A_n\}$ is bounded in $H^{s_1}$. Note that $\mbox{dist}_{H^{s_1}}(\tilde{u}_n, A_n)=o_n(1)$, then $\{\tilde{u}_n\}$ is bounded in $H^{s_1}(\R^N)$. In addition, since $A_n$ is compact for any $n\in \N$, then there exists $v_n \in A_n$ such that
$$
\mbox{dist}_{H^{s_1}}(\tilde{u}_n, A_n)=\|\tilde{u}_n-v_n\|_{H^{s_1}}=o_n(1).
$$
Applying again Lemma \ref{coercive}, we then get that
$$
\|\tilde{u}_n\|_{H^{s_1}} \geq \|v_n\|_{H^{s_1}} -\|\tilde{u}_n-v_n\|_{H^{s_1}} \geq \frac{1}{C_2}+o_n(1).
$$
Therefore, the claim follows.

We next show that $\{u_n\} \subset P(c)$ is a Palais-Smale sequence for $E$ restricted on $S(c)$ at the level $\gamma_{\mathcal{G}}(c)$. In the following, we denote by $\|\cdot\|_{*}$ the dual norm of $(T_u S(c))^*$. Observe that
\begin{align*}
\|dE(u_n)\|_*=\sup_{\psi \in T_{u_n}S(c), \|\psi\|_{H^{s_1}}\leq 1}|dE(u_n)[\psi]|=\sup_{\psi \in T_{u_n}S(c), \|\psi\|_{H^{s_1}}\leq 1}|dE(u_n)[(\psi_{\frac{1}{t_n}})_{t_n}]|.
\end{align*}
By straightforward calculations, we can find that the mapping $T_uS(c) \to T_{u_{t_u}}S(c)$ define by $\psi \mapsto \psi_{t_u}$ is an isomorphism. Moreover, we have that $dF(u)[\psi]=dE(u_{t_u})[\psi_{t_u}]$ for any $u \in S(c)$ and $\psi \in T_uS(c)$. As a consequence, we get that
$$
\|dE(u_n)\|_*=\sup_{\psi \in T_{u_n}S(c), \|\psi\|_{H^{s_1}}\leq 1}|dF(\tilde{u}_n)[\psi_{\frac{1}{t_n}}]|.
$$
Since $\{\tilde{u}_n\} \subset S(c)$ is a Palais-Smale sequence for $F$ at the level $\gamma_{\mathcal{G}}(c)$, we then  apply the claim to deduce that $\{u_n\} \subset P(c)$ is a Palais-Smale sequence for $E$ restricted on $S(c)$ at the level $\gamma_{\mathcal{G}}(c)$. Thus the proof is completed.
\end{proof}

\begin{lem} \label{pss}
Let $N \geq 1$, $0<s_2<s_1<1$ and $p \geq 2 +\frac{4s_1}{N}$. Then there exists a Palais-Smale sequence $\{u_n\} \subset P(c)$ for $E$ restricted on $S(c)$ at the level $\gamma(c)$ for any $c>c_0$.
\end{lem}

\begin{proof}
Let $B=\emptyset$ and $\mathcal{G}$ be all singletons in $\mathcal{S}(c)$ if $p=2+\frac{4s_1}{N}$ and all singletons in $S(c)$ if $p>2+\frac{4s_1}{N}$, where $\mathcal{S}(c)$ is defined by \eqref{sc}. Therefore, from \eqref{ming}, there holds that
$$
\gamma_{\mathcal{G}}(c)=\inf_{u \in \mathcal{S}(c)} \sup_{t>0} E(u_t) \,\,\,\mbox{if} \,\, p=2+\frac{4s_1}{N}
$$
and
$$
\gamma_{\mathcal{G}}(c)=\inf_{u \in S(c)} \sup_{t>0} E(u_t) \,\,\,\mbox{if} \,\, p>2+\frac{4s_1}{N}.
$$
We next prove that $\gamma_{\mathcal{G}}(c)=\gamma(c)$. For simplicity, we only consider the case $p>2+\frac{4s_1}{N}$. From Lemma \ref{monotonicity}, we know that, for any $u \in S(c)$, there exists a unique $t_u>0$ such that $u_{t_u} \in P(c)$ and $E(u_{t_u})=\displaystyle\max_{t >0}E(u_t)$. This then implies that
$$
\inf_{u \in S(c)} \sup_{t>0} E(u_t)  \geq \inf_{u \in P(c)} E(u).
$$
On the other hand, for any $u \in P(c)$, we have that $E(u)=\displaystyle\max_{t >0}E(u_t)$. This then gives that
$$
\inf_{u \in S(c)} \sup_{t>0} E(u_t)  \leq \inf_{u \in P(c)} E(u).
$$
Accordingly, we derive that $\gamma_{\mathcal{G}}(c)=\gamma(c)$. It then follows from Lemma \ref{ps} that the result of this lemma holds true and the proof is completed.
\end{proof}

\begin{lem} \label{lagrange}
Let $N \geq 1$, $0<s_2<s_1<1$ and $p \geq 2 +\frac{4s_1}{N}$. Let $u_c \in S(c)$ be a solution to the equation
\begin{align} \label{fequ1}
(-\Delta)^{s_1} u_c +(-\Delta)^{s_2} u_c + \lambda_c u_c=|u_c|^{p-2} u_c, \quad c>c_0.
\end{align}
Then there exists a constant $c_1>0$ such that $\lambda_c>0$ for any $c_0<c<c_1$. In particular, if $N=1$ and $2s_2 \geq 1$ or $N \geq 1$, $2s_2<N$ and $2< p \leq \frac{2N}{N-2s_2}$, then $c_1=\infty$.
\end{lem}

\begin{proof}
Since $u_c$ is a solution to \eqref{fequ1}, then $Q(u_c)=0$, see Lemma \ref{pohozaev}. This means that
\begin{align} \label{ph111}
s_1\int_{\R^N}|(-\Delta)^{\frac{s_1}{2}} u_c|^2 \,dx +s_2\int_{\R^N}|(-\Delta)^{\frac{s_2}{2}} u_c|^2 \,dx =\frac{N(p-2)}{2p} \int_{\R^N}|u_c|^p\,dx.
\end{align}
Multiplying \eqref{fequ1} by $u_c$ and integrating on $\R^N$, we get that
\begin{align} \label{integrate}
\int_{\R^N}|(-\Delta)^{\frac{s_1}{2}} u_c|^2 \,dx +\int_{\R^N}|(-\Delta)^{\frac{s_2}{2}} u_c|^2 \,d +\lambda_c \int_{\R^N}|u_c|^2 \,dx =\int_{\R^N}|u_c|^p \,dx .
\end{align}
Combining \eqref{ph111} and \eqref{integrate}, we have that
\begin{align} \label{lc}
 \hspace{-0.5cm}\lambda_c c = \left(\frac{2ps_1}{N(p-2)}-1\right) \int_{\R^N}|(-\Delta)^{\frac{s_1}{2}} u_c|^2 \,dx + \left(\frac{2ps_2}{N(p-2)}-1\right) \int_{\R^N}|(-\Delta)^{\frac{s_2}{2}} u_c|^2 \,dx.
\end{align}
If $N=1$ and $2s_2\geq 1$ or $N \geq 1$, $2s_2<N$ and $2< p \leq \frac{2N}{N-2s_2}$, then
$$
\frac{2ps_1}{N(p-2)}>\frac{2ps_2}{N(p-2)} \geq 1.
$$
This indicates that $\lambda_c>0$, by \eqref{lc}. In this case, we choose $c_1=\infty$. We now treat the case  $N \geq 1$, $N >2s_2$ and $p >\frac{2N}{N-2s_2}$. In virtue of \eqref{gn} and \eqref{ph111}, we derive that
$$
s_1\int_{\R^N}|(-\Delta)^{\frac{s_1}{2}} u_c|^2 \,dx +s_2\int_{\R^N}|(-\Delta)^{\frac{s_2}{2}} u_c|^2 \,dx  \leq \widetilde{C}_{N,p,s_1} c^{\frac p 2 -\frac{N(p-2)}{4s_1}}\left(\int_{\R^N}|(-\Delta)^{\frac{s_1}{2}} u_c|^2 \,dx \right)^{\frac{N(p-2)}{4s_1}},
$$
where $\widetilde{C}_{N,p,s_1}>0$ is defined by
$$
\widetilde{C}_{N,p,s_1}:=\frac{N(p-2)C_{N,p,s_1}}{2p}.
$$
If $p>2+\frac{4s_1}{N}$, from the inequality above, it then leads to $\|(-\Delta)^{{s_1}/{2}} u_c\|_2 \to \infty$ as $c \to 0^+$. In addition, by interpolation inequality, there holds that
\begin{align*}
 \int_{\R^N}|(-\Delta)^{\frac{s_2}{2}} u_c|^2 \,dx \leq \left(\int_{\R^N}|(-\Delta)^{\frac{s_1}{2}} u_c|^2 \,dx \right)^{\frac{s_2}{s_1}} \left(\int_{\R^N}|u_c|^2\,dx \right)^{\frac{s_1-s_2}{s_1}}.
\end{align*}
Due to
$$
\frac{2ps_1}{N(p-2)}>1, \quad s_2<s_1,
$$
we then conclude from \eqref{lc} that $\lambda_c>0$ if $c>0$ small enough. We now treat the case $p=2+\frac{4s_1}{N}$. In this case, from \eqref{ph111} and \eqref{integrate}, we see that
$$
\lambda_c c=(s_1-s_2) \int_{\R^N} |(-\Delta)^{\frac {s_1}{2}} u|^2 \,dx-\frac{N(s_1-s_2)-2s_1s_2}{N+2s_1} \int_{\R^N} |u|^{2+\frac{4s_1}{N}} \,dx.
$$
If $N \geq 1$, $N \geq 2s_2$ and $p>\frac{2N}{N-2s_2}$, i.e. $N(s_1-s_2) > 2s_1 s_2$, by \eqref{gn}, then
$$
\lambda_c c \geq \left((s_1-s_2)-\frac{N(s_1-s_2)-2s_1s_2}{N}\left(\frac{c}{c_{N,s_1}}\right)^{\frac{2s_1}{N}}\right)\int_{\R^N} |(-\Delta)^{\frac {s_1}{2}}u|^2 \,dx,
$$
from which we get that $\lambda_c>0$ if $c_{N,s_1}<c<c_1$, where $c_1>0$ is defined by
$$
c_1:=\left(\frac{N(s_1-s_2)}{N(s_1-s_2)-2s_1s_2}\right)^{\frac{N}{2s_1}} c_{N,s_1}.
$$
This completes the proof.
\end{proof}

\begin{lem} \label{nonincreasing}
Let $N \geq 1$, $0<s_2<s_1<1$ and $p \geq 2 +\frac{4s_1}{N}$. Then the function $c \mapsto \gamma(c)$ is nonincreasing on $(c_0, \infty)$.
\end{lem}

\begin{proof}
For any $0<c_2<c_1$, we shall prove that $\gamma(c_1) \leq \gamma(c_2)$. By the definition of $\gamma(c)$, we first have that, for any $\eps>0$, there exists $u \in P(c_2)$ such that
\begin{align} \label{non1}
E(u) \leq \gamma(c_2) + \frac{\eps}{2}.
\end{align}
Let $\chi \in C_0^{\infty}(\R^N, [0,1])$ be a cut-off function such that $\chi(x) =1$ for $|x| \leq 1$ and $\chi(x)=0$ for $|x|\geq 2$. For $\delta>0$ small, we define $u^{\delta}(x):=u(x) \chi(\delta x)$ for $x \in \R^N$. It is easy to check that $u^{\delta} \to u$ in $H^{s_1}(\R^N)$ as $\delta \to 0^+$. Since $Q(u)=0$, we then get that $(u^{\delta})_{t_{u^{\delta}}} \to u$ in $H^{s_1}(\R^N)$ as $\delta \to 0^+$, where $t_{u^{\delta}}>0$ is determined by Lemma \ref{monotonicity} such that $Q((u^{\delta})_{t_{u^{\delta}}})=0$. Therefore, there exists a constant $\delta>0$ small such that
\begin{align} \label{non2}
E((u^{\delta})_{t_{u_{\delta}}}) \leq E(u) + \frac{\eps}{4}.
\end{align}
Let $v \in C^{\infty}_0(\R^N)$ be such that $\mbox{supp}\,v \subset B(0, 1+2/\delta) \backslash B(0, 2/\delta)$ and set
$$
\tilde{v}^{\delta}:=\frac{\left(c_1-\|u^{\delta}\|_2^2\right)^{\frac 12}}{\|v\|_2} v.
$$
For $0<t<1$, we define $w^{\delta}_t:=u^{\delta} + (\tilde{v}^{\delta})_t$, where $ (\tilde{v}^{\delta})_t(x):= t^{N/2}\tilde{v}^{\delta}(tx)$ for $x\in \R^N$. Observe that $\mbox{supp}\, u_{\delta} \cap \mbox{supp}\, (\tilde{v}^{\delta})_t =\emptyset$ for any $0<t<1$, then $\|w^{\delta}_t\|_2^2=c_1$. It is not hard to verify that $ (\tilde{v}^{\delta})_t \to 0$ in $\dot{H}^{s_1}(\R^N)$ as $t \to 0^+$. Hence we deduce that there exist $t, \delta>0$ small such that
\begin{align}\label{non3}
\max_{\lambda>0}E(((\tilde{v}^{\delta})_t)_{\lambda}) \leq \frac{\eps}{4}.
\end{align}
Consequently, using \eqref{non1}-\eqref{non3}, we have that
\begin{align*}
\gamma(c_1) \leq \max_{\lambda>0} E((w^{\delta}_t)_{\lambda})&=\max_{\lambda>0}\left(E((u^{\delta})_{\lambda}) + E(((\tilde{v}^{\delta})_t)_{\lambda}) \right) \\
&\leq E((u^{\delta})_{t_{u_{\delta}}})  + \frac{\eps}{4}\leq E(u)+ \frac{\eps}{2} \leq \gamma(c_2) +\eps.
\end{align*}
Since $\eps>0$ is arbitrary, then $\gamma(c_1) \leq \gamma(c_2)$. Thus the proof is completed.
\end{proof}

\begin{lem}\label{decreasing}
Let $N \geq 1$, $0<s_2<s_1<1$ and $p \geq 2 +\frac{4s_1}{N}$. If there exists $u \in S(c)$ with $E(u)=\gamma(c)$ satisfying the equation
\begin{align} \label{fequ2}
(-\Delta)^{s_1} u +(-\Delta)^{s_2} u + \lambda u=|u|^{p-2} u,
\end{align}
then $\lambda \geq 0$. If $\lambda>0$, then the function $c \mapsto \gamma(c)$ is strictly decreasing in a right neighborhood of $c$.  If $\lambda<0$, then the function $c \mapsto \gamma(c)$ is strictly increasing in a left neighborhood of $c$.
\end{lem}

\begin{proof}
For any $t_1, t_2>0$, we set $u_{t_1,t_2}(x):=t_1^{1/2}t_2^{N/2}u(t_2x)$ for $x \in \R^N$. Then we find that $\|u_{t_1,t_2}\|_2^2 =t_1c$. Define
\begin{align*}
\alpha(t_1,t_2):=E(u_{t_1,t_2})=\frac{t_1^2t_2^{2s_1}}{2}\int_{\R^N}|(-\Delta)^{\frac{s_1}{2}} u|^2\,dx +
\frac{t_1^2t_2^{2s_2}}{2}\int_{\R^N}|(-\Delta)^{\frac{s_s}{2}} u|^2\,dx-\frac{t_1^pt_2^{\frac{N}{2}(p-2)}}{p} \int_{\R^N}|u|^p\,dx.
\end{align*}
We compute that
\begin{align*}
\frac{\partial}{\partial t_1} \alpha(t_1, t_2)={t_1t_2^{2s_1}}\int_{\R^N}|(-\Delta)^{\frac{s_1}{2}} u|^2\,dx+
{t_1t_2^{2s_2}}\int_{\R^N}|(-\Delta)^{\frac{s_2}{2}} u|^2\,dx  -{t_1^{p-1}t_2^{\frac{N}{2}(p-2)}} \int_{\R^N}|u|^p\,dx.
\end{align*}
Note that $u_{t_1, t_2} \to u$ in $H^{s_1}(\R^N)$ as $(t_1, t_2) \to (1, 1)$ and $E'(u)u=-\lambda c$. If $\lambda>0$, then there exists a constant $\delta>0$ small such that, for any $(t_1, t_2) \in (1, 1+\delta) \times [1-\delta, 1+\delta]$,
$$
\frac{\partial}{\partial t_1} \alpha(t_1, t_2) <0.
$$
This leads to $\alpha(t_1, t_2) <\alpha(1, t_2)$ for any $(t_1, t_2) \in (1, 1+\delta) \times [1-\delta, 1+\delta]$. Observe that
\begin{align*}
\frac{\partial}{\partial t_2} \alpha(t_1, t_2)&=s_1{t_1^2t_2^{2s_1-1}}\int_{\R^N}|(-\Delta)^{\frac{s_1}{2}} u|^2\,dx+
s_2{t_1^2t_2^{2s_2-1}}\int_{\R^N}|(-\Delta)^{\frac{s_2}{2}} u|^2\,dx \\
& \quad -\frac{N(p-2)}{2p}{t_1^p t_2^{\frac{N}{2}(p-2)-1}} \int_{\R^N}|u|^p\,dx \\
&=t_1^2\frac{Q(u_{t_2})}{t_2}
\end{align*}
and
\begin{align*}
\frac{\partial^2}{\partial t_2^2} \alpha(t_1, t_2)&=s_1(2s_1-1){t_1^2t_2^{2s_1-2}}\int_{\R^N}|(-\Delta)^{\frac{s_1}{2}} u|^2\,dx+s_2(2s_2-1){t_1^2t_2^{2s_2-2}}\int_{\R^N}|(-\Delta)^{\frac{s_2}{2}} u|^2\,dx \\
& \quad -\frac{N(p-2)(N(p-2)-2)}{4p}{t_1^p t_2^{\frac{N}{2}(p-2)-2}} \int_{\R^N}|u|^p\,dx.
\end{align*}
Then we have that
$$
\frac{\partial}{\partial t_2} \alpha(t_1, t_2) {\mid}_{(1,1)}=0, \quad \frac{\partial^2}{\partial t_2^2} \alpha(t_1, t_2) {\mid}_{(1,1)}<0.
$$
For $\eps>0$ small, by the implicit function theorem, then there exists a continuous function $g: [1-\eps, 1+\eps] \to \R $ with $g(1)=1$ such that $Q(u_{1+\eps, g(1+\eps)})=0$. Therefore, we conclude that
$$
\gamma((1+\eps)c) \leq \alpha(1+\eps, g(1+\eps))<\alpha(1, g(1+\eps)) \leq \alpha(1,1)=E(u)=\gamma(c),
$$
where we used the fact $u \in P(c)$. If $\lambda<0$, we can similarly obtain the desired result. This jointly with Lemma \ref{nonincreasing} implies that $\lambda \geq 0$. Thus the proof is completed.
\end{proof}

\begin{lem}\label{ladecreasing}
Let $N \geq 1$, $0<s_2<s_1<1$ and $p \geq 2 +\frac{4s_1}{N}$. Then the function $c \mapsto \gamma(c)$ is strictly decreasing on $(c_0, c_1)$, where the constant $c_1>0$ is determined in Lemma \ref{lagrange}.
\end{lem}

\begin{proof}
The proof follows directly from Lemmas \ref{lagrange} and \ref{decreasing}.
\end{proof}

We are now ready to prove Theorem \ref{thm2}.

\begin{proof}[Proof of Theorem \ref{thm2}]
In light of Lemma \ref{pss}, we first obtain that there exists a Palais-Smale sequence $\{u_n\} \subset P(c)$ for $E$ restricted on $S(c)$ at the level $\gamma(c)$. From Lemma \ref{coercive}, we have that $\{u_n\}$ is bounded in $H^{s_1}(\R^N)$. Reasoning as the proof of \cite[Lemma 3]{BeLi}, we are able to deduce that $u_n \in H^{s_1}(\R^N)$ satisfies the following equation,
\begin{align} \label{fequ3}
(-\Delta)^{s_1} u_n +(-\Delta)^{s_2} u_n + \lambda_n u_n=|u_n|^{p-2} u_n+o_n(1),
\end{align}
where
$$
\lambda_n:=\frac{1}{c} \left(\int_{\R^N}|u_n|^p\,dx -\int_{\R^N}|(-\Delta)^{\frac{s_1}{2}} u_n|^2\, dx-\int_{\R^N}|(-\Delta)^{\frac{s_2}{2}} u_n|^2\,dx \right).
$$
We claim that $\{u_n\} \subset H^{s_1}(\R^N)$ is non-vanishing. Otherwise, by using \cite[Lemma I.1]{Li}, we have that $\|u_n\|_p=o_n(1)$. Since $Q(u_n)=0$, then there holds that $E(u_n)=o_n(1)$. This is impossible, because of $\gamma(c)>0$ for any $c>c_0$, see Lemma \ref{coercive}. As a result, we know that there exists a sequence $\{y_n\} \subset \R^N$ such that $u_n(\cdot+y_n) \wto u$  in $H^{s_1}(\R^N)$ as $n \to \infty$ and $u\neq 0$. Since $\{u_n\}$ is bounded in $H^{s_1}(\R^N)$, then $\{\lambda_n\} \subset \R$ is bounded. Hence, there exists a constant $\lambda\in\R$ such that $\lambda_n \to \lambda$ in $\R$ as $n \to \infty$. Therefore, from \eqref{fequ3}, we get that
\begin{align} \label{fequ4}
(-\Delta)^{s_1} u +(-\Delta)^{s_2} u + \lambda u=|u|^{p-2} u.
\end{align}
This results in $Q(u)=0$, see Lemma \ref{pohozaev}. We next prove that $u \in S(c)$. To do this, let us define $w_n=u_n-u(\cdot-y_n)$. It is immediate to see that $w_n(\cdot+y_n) \wto 0$ in $H^{s_1}(\R^N)$ as $n \to \infty$. In addition, there holds that
\begin{align*}
\|(-\Delta)^{\frac{s_i}{2}} u_n\|^2_2=\|(-\Delta)^{\frac{s_i}{2}} w_n\|^2_2+\|(-\Delta)^{\frac{s_i}{2}} u\|^2_2+o_n(1) \quad \mbox{for}\,\, i=1,2
\end{align*}\
and
$$
\|u_n\|_p^p=\|w_n\|_p^p+\|u\|_p^p+o_n(1).
$$
Thus we have that
\begin{align} \label{bl}
\begin{split}
\gamma(c)=E(u_n)+o_n(1) =E(w_n) +E(u)+o_n(1) \geq E(w_n) +\gamma(\|u\|_2^2)+o_n(1),
\end{split}
\end{align}
and
\begin{align} \label{bl1}
Q(w_n)=Q(w_n)+Q(u)=Q(u_n)+o_n(1)=o_n(1).
\end{align}
In view of \eqref{bl1}, then
$$
0 \leq E(w_n)-\frac{2p}{N(p-2)}Q(w_n)=E(w_n) +o_n(1).
$$
Since $0<\|u\|_2^2 \leq c$, by using \eqref{bl} and Lemma \ref{nonincreasing}, we then have that $\gamma(\|u\|_2^2)=\gamma(c)$. This together with Lemmas \ref{lagrange} and \ref{ladecreasing} gives rise to $u \in S(c)$. As a consequence, we get that $\|w_n\|_p=o_n(1)$. From \eqref{fequ3} and \eqref{fequ4}, it then follows that $\|(-\Delta)^{\frac{s_i}{2}} w_n\|^2_2=o_n(1)$ for $i=1,2$. Therefore, we are able to derive that $\gamma(c)=E(u)$. Thus the proof is completed.
\end{proof}

\begin{proof}[Proof of Theorem \ref{thm6}]
Let $u \in S(c)$ be a ground state solution to \eqref{fequ}-\eqref{mass} at the level $\gamma(c)$. We claim that
$$
|u| \in P(c), \quad \|(-\Delta)^{\frac{s_1}{2}} |u|\|_2 = \|(-\Delta)^{\frac{s_1}{2}} u \|_2, \quad \|(-\Delta)^{\frac{s_2}{2}} |u|\|_2 = \|(-\Delta)^{\frac{s_2}{2}} u \|_2.
$$
Indeed, we first observe that $E(|u|) \leq E(u)$ and $Q(|u|) \leq Q(u)=0$. Then, by Lemma \ref{monotonicity}, there exists a unique constant $0<t_{|u|} \leq 1$ such that $|u|_{t_{|u|}} \in P(c)$. Therefore, we conclude that
\begin{align*}
\gamma(c) \leq E(|u|_{t_{|u|}})&=E(|u|_{t_{|u|}})-\frac{2}{N(p-2)}Q(|u|_{t_{|u|}})\\
&=\frac{N(p-2)-4s_1}{2N(p-2)}\|(-\Delta)^{\frac{s_1}{2}} \left(|u|_{t_{|u|}}\right)\|_2^2 +\frac{N(p-2)-4s_2}{2N(p-2)}\|(-\Delta)^{\frac{s_2}{2}} \left( |u|_{t_{|u|}}\right)\|_2^2\\
&=t_{|u|}^{2s_1}\frac{N(p-2)-4s_1}{2N(p-2)}\|(-\Delta)^{\frac{s_1}{2}} |u|\|_2^2 +t_{|u|}^{2s_2}\frac{N(p-2)-4s_2}{2N(p-2)}\|(-\Delta)^{\frac{s_2}{2}} |u|\|_2^2\\
& \leq \frac{N(p-2)-4s_1}{2N(p-2)}\|(-\Delta)^{\frac{s_1}{2}} |u|\|_2^2 +\frac{N(p-2)-4s_2}{2N(p-2)}\|(-\Delta)^{\frac{s_2}{2}} |u|\|_2^2\\
& \leq \frac{N(p-2)-4s_1}{2N(p-2)}\|(-\Delta)^{\frac{s_1}{2}} u\|_2^2 +\frac{N(p-2)-4s_2}{2N(p-2)}\|(-\Delta)^{\frac{s_2}{2}} u\|_2^2 \\
&=E(u)-\frac{2}{N(p-2)}Q(u)=E(u)=\gamma(c).
\end{align*}
This leads to $t_{|u|}=1$ and
$$
\|(-\Delta)^{\frac{s_1}{2}} |u|\|_2 = \|(-\Delta)^{\frac{s_1}{2}} u \|_2, \quad \|(-\Delta)^{\frac{s_2}{2}} |u|\|_2 = \|(-\Delta)^{\frac{s_2}{2}} u \|_2.
$$
Then the claim follows. Hence we have that $u=e^{i \theta} |u|$ for some $\theta \in \mathbb{S}$. From the discussion above, we know that $|u| \in S(c)$ is also a ground state solution to \eqref{fequ}-\eqref{mass} at the level $\gamma(c)$. Let us now denote by $|u|^{\ast}$ the symmetric-decreasing rearrangement of $|u|$. Similarly, we are able to show that
$$
|u|^{\ast} \in P(c) ,\quad \|(-\Delta)^{\frac{s_1}{2}} |u|^{\ast}\|_2 = \|(-\Delta)^{\frac{s_1}{2}} |u| \|_2, \quad \|(-\Delta)^{\frac{s_2}{2}} |u|^{\ast}\|_2 = \|(-\Delta)^{\frac{s_2}{2}} |u| \|_2.
$$
From \cite[Proposition 3]{FSS}, we then derive that $|u|=\rho(|\cdot-x_0|)$ for some $x_0 \in \R^N$, where $\rho$ is a decreasing function. Thus the proof is completed.
\end{proof}

\section{Multiplicity of bound state solutions} \label{section4}

In this section, we aim to prove Theorem \ref{thm3}. To begin with, we need to fix some notations. We denote by $\sigma : H^{s_1}(\R^N) \to H^{s_1}(\R^N)$ the transformation $\sigma(u)=-u$. A set $A \subset H^{s_1}(\R^N)$ is called $\sigma$-invariant if $\sigma(A)=A$. Let $Y \subset H^{s_1}(\R^N)$. A homotopy $\eta: [0,1] \times Y \to Y$ is called $\sigma$-equivariant if $\eta(t,\sigma(u))=\sigma(\eta(t,u))$ for any $(t, u) \in [0,1] \times Y$.

\begin{defi}\label{homotopy1} \cite[Definition 7.1]{Gh}
Let $B$ be a closed $\sigma$-invariant subset of a set $Y \subset H^{s_1}(\R^N)$. We say that a class $\mathcal{F}$ of compact subsets of $Y$ is a $\sigma$-homotopy stable family with the closed boundary $B$ provided that
\begin{enumerate}
\item [\textnormal{(i)}] every set in $\mathcal{F}$ is $\sigma$-invariant;
\item [\textnormal{(ii)}] every set in $\mathcal{F}$ contains $B$;
\item [\textnormal{(iii)}] for any $A \in \mathcal{F}$ and any $\sigma$-equivariant homotopy $\eta \in C([0, 1] \times Y, Y)$ satisfying $\eta(t, x)=x$ for all $(t, x) \in (\{0\} \times Y) \cup([0, 1] \times B)$, then $\eta(\{1\} \times A) \in \mathcal{G}$.
\end{enumerate}
\end{defi}

\begin{lem} \label{ps1}
Let $\mathcal{F}$ be a $\sigma$-homotopy stable family of compact subsets of $S_{rad}(c)$ with a close boundary $B$. Let
$$
\gamma_{\mathcal{F}}(c):= \inf_{A\in \mathcal{F}}\max_{u\in A}F(u).
$$
Suppose that $B$ is contained in a connected component of $P_{rad}(c)$ and $ \max \{\sup E(B),0\}<\gamma_{\mathcal{F}}(c)<\infty$. Then there exists a Palais-Smale sequence $\{u_n\} \subset P_{rad}(c)$ for $E$ restricted on $S_{rad}(c)$ at the level $\gamma_{\mathcal{F}}(c)$.
\end{lem}
\begin{proof}
By applying \cite[Theorem 7.2]{Gh}, the proof is almost identical to the one of Lemma \ref{ps}, then we omit the proof.
\end{proof}

\begin{defi}\label{genus}
For any closed $\sigma$-invariant set $A \subset H^{s_1}(\R^N)$, the genus of $A$ is defined by
$$
\textnormal{Ind}(A):= \min \{n \in \N^+ : \exists \,\, \varphi : A \rightarrow \R^n\backslash \{0\}, \varphi \,\, \text{is continuous and odd}\}.
$$
If there is no $\varphi$ as described above, we set $\textnormal{Ind}(A)= \infty.$ If $A=\emptyset$, we set $\textnormal{Ind}(A)=0$.
\end{defi}

Let $\mathcal{A}$ be a family of compact and $\sigma$-invariant sets contained in $S_{rad}(c)$. For any $k \in \N^+
$, we now define
\begin{align*}
\beta_k(c):=\inf_{A \in \mathcal{A}_{k}}\sup_{u \in A}E(u),
\end{align*}
where the set $\mathcal{A}_k$ is defined by
$$
\mathcal{A}_{k}:=\{A \in \mathcal{A}:\textnormal{Ind}(A) \geq k\}.
$$
\begin{lem} \label{akne}
\begin{enumerate}
\item [\textnormal{(i)}] If $p=2+\frac{4s_1}{N}$, then, for any $k \in \N^+$, there exists a constant $c_k>c_{N,s_1}$ such that $\mathcal{A}_k \neq \emptyset$ for any $c \geq c_k$, where $c_{N, s_1}>0$ is determined in Theorem \ref{thm1}.
\item [\textnormal{(ii)}] If $p >2+\frac{4s_1}{N}$, then, for any $k \in \N^+$, $\mathcal{A}_k \neq \emptyset$ for any $c >0$.
\item [\textnormal{(iii)}] For any $k \in \N^+$, there holds that $\beta_{k+1}(c) \geq \beta_{k}(c)>0$.
\end{enumerate}
\end{lem}

\begin{proof}
Let us first prove $\mbox{(i)}$. For any $k \in \N^+$ and $V_k \subset H^{s_1}_{rad}(\R^N)$ be such that $\mbox{dim} V_k=k$, we define $SV_k(c):=V_k \cap S(c)$. By basic properties of genus, see \cite[Theorem 10.5]{AM}, we have that $\mbox{Ind}(SV_k(c))=k$. Note that all norms are equivalent in finite dimensional subsequence of $H^{s_1}(\R^N)$, then there exists a constant $c_k>c_{N, s_1}$ large enough such that, for any $u \in SV_k(c)$, there holds that $u \in \mathcal{S}(c)$, where $\mathcal{S}(c)$ is defined by \eqref{sc}. It follows that $\sup_{t>0}E(u_t)<\infty$. Thus we apply Lemma \ref{monotonicity} to conclude that, for any $u \in SV_k(c)$ with $c \geq c_k$, there exists a unique $t_u>0$ such that $u_{t_u} \in P(c)$. We now define a mapping $\varphi: SV_k(c) \to P_{rad}(c)$ by $\varphi(u)=u_{t_u}$. It is easy to see that $\varphi$ is continuous and odd. By means of \cite[Lemma 10.4]{AM}, we then obtain that  $\mbox{Ind}(\varphi(SV_k(c)))\geq \mbox{Ind}(SV_k(c))=k$. This then suggests that $\mathcal{A}_k \neq \emptyset$. The assertion $\mbox{(ii)}$ can be achieved by using similar arguments. We now prove $\mbox{(iii)}$. Observe that, for any $k \in \N^
+$, $\mathcal{A}_{k+1} \subset \mathcal{A}_k$, then $\beta_{k+1}(c) \geq \beta_{k}(c)$. In addition, by the definition of $\beta_k(c)$, we have that, for any $k \in \N^+$, $\beta_k(c) \geq \gamma(c)>0$. Thus the proof is completed.
\end{proof}

At this point, we are able to present the proof of Theorem \ref{thm3}.

\begin{proof}[Proof of Theorem \ref{thm3}]
Let us first prove the assertion $(\textnormal{i})$. Thanks to Lemma \ref{akne}, we know that, for any $k \in \N^+$, $0<\beta_k(c)<\infty$. From Lemma \ref{ps1}, we have that, for any $k\in \N^+$, there exists a Palais-Smale sequence $\{u_n^k\} \subset P_{rad}(c)$ for $E$ restricted on $S_{rad}(c)$ at the level $\beta_k(c)$. Arguing as the proof of Theorem \ref{thm2}, we can derive that $u_n^k \in H^{s_1}(\R^N)$ satisfies the following equation,
\begin{align} \label{fequ31}
(-\Delta)^{s_1} u_n^k +(-\Delta)^{s_2} u_n^k + \lambda_n^k u_n^k=|u_n^k|^{p-2} u_n^k+o_n(1),
\end{align}
where
$$
\lambda_n^k:=\frac{1}{c} \left(\int_{\R^N}|u_n^k|^p\,dx -\int_{\R^N}|(-\Delta)^{\frac{s_1}{2}} u_n^k|^2\, dx-\int_{\R^N}|(-\Delta)^{\frac{s_2}{2}} u_n^k|^2\,dx \right).
$$
In addition, there exist a constant $\lambda_k \in \R$ such that $\lambda_n^k \to \lambda_k$ as $n \to \infty$ and a nontrivial $u^k \in H^{s_1}(\R^N)$ such that $u_n^k \wto u^k$ in $H_{rad}^{s_1}(\R^N)$ as $n \to \infty$. As a consequence, we deduce that $u^k \in H^{s_1}(\R^N)$ satisfies the equation
\begin{align} \label{fequ41}
(-\Delta)^{s_1} u^k +(-\Delta)^{s_2} u^k + \lambda^k u^k=|u^k|^{p-2} u^k.
\end{align}
Taking into account the fact that $H^{s_1}_{rad}(\R^N) \hookrightarrow L^p(\R^N)$ is compact for $N \geq 2$, we then get that $u_n^k \to u^k$ in $L^p(\R^N)$ as $n \to \infty$. Since $Q(u_n^k)=Q(u)=0$, then
$$
\|(-\Delta)^{\frac{s_i}{2}} u_n^k-(-\Delta)^{\frac{s_i}{2}}u^k\|_2=o_n(1), \quad i=1,2.
$$
In view of Lemma \ref{lagrange}, we know that $\lambda_k>0$ for any $c_0<c<c_1$. Applying \eqref{fequ31} and \eqref{fequ41}, we then have that $u_n^k \to u^k$ in $H_{rad}^{s_1}(\R^N)$ as $n \to \infty$. This indicates that $u^k \in S(c)$ is a solution to \eqref{fequ}-\eqref{mass} and $E(u^k)=\beta_k(c)$. Reasoning as the proof of \cite[Proposition 9.33]{Ra}, we can derive that $\beta_k(c) \to \infty$ as $k \to \infty$. Note that if $p=2+\frac{4s_1}{N} \leq \frac{2N}{N-2s_2}$, then $c_1=\infty$, see Lemma \ref{lagrange}. Hence, by a similar way, the assertion $(\textnormal{ii})$ follows. The proof is completed.
\end{proof}

\section{Properties of the function $c \mapsto \gamma(c)$} \label{section5}

In this section, our goal is to prove Theorem \ref{thm4}. We first have the following result.

\begin{lem}\label{continuous}
Let $N \geq 1$, $0<s_2<s_1<1$ and $p\geq 2+\frac{4s_1}{N}$. Then the function $c \mapsto \gamma(c)$ is continuous for any $c>c_0$.
\end{lem}
\begin{proof}
For simplicity, we only consider the case $p>2+\frac{4s_1}{N}$. Let $c>c_0$ and $\{c_n\} \subset (c_0, \infty)$ satisfying $c_n \to c$ as $n \to \infty$, we shall prove that $\gamma(c_n) \to \gamma(c)$ as $n \to \infty$. From the definition of $\gamma(c)$, we know that, for any $\eps>0$, there exists $v \in P(c)$ such that
$$
E(v) \leq \gamma(c) + \eps/2.
$$
Define
$$
v_n:=\left(\frac{c_n}{c}\right)^{\frac 12} v \in S(c_n).
$$
It is not hard to verify that $v_n \to v$ in $H^{s_1}(\R^N)$ as $n \to \infty$. Therefore, by Lemma \ref{monotonicity}, we can deduce that
\begin{align*}
\gamma(c_n) \leq \max_{\lambda>0} E((v_n)_{\lambda}) \leq \max_{\lambda>0} E(v_{\lambda}) + \frac{\eps}{2}=E(v)+\frac{\eps}{2} \leq \gamma(c) +\eps.
\end{align*}
This implies that
$$
\displaystyle \limsup_{n\to \infty}\gamma(c_n) \leq \gamma(c).
$$
By a similar way, we can prove that
$$
\gamma(c) \leq \displaystyle\liminf_{n\to \infty}\gamma(c_n).
$$
Thus proof is completed.
\end{proof}

\begin{lem} \label{limit1}
Let $N \geq 1$, $0<s_2<s_1<1$ and $p\geq 2+\frac{4s_1}{N}$. Then $\gamma(c) \to \infty$ as $c \to c_0^+$.
\end{lem}
\begin{proof}
Let $\{c_n\} \subset (c_0, \infty)$ satisfy $c_n \to c_0$ as $n \to \infty$. In view of Theorem \ref{thm2}, we have that there exists $\{u_n\}\subset P(c_n)$ such that $E(u_n)=\gamma(c_n)$. Record that
\begin{align} \label{ide111}
\begin{split}
\gamma(c_n)=E(u_n)&=E(u_n)-\frac{2}{N(p-2)}Q(u_n) \\
&=\frac{N(p-2)-4s_1}{2N(p-2)}\int_{\R^N}|(-\Delta)^{\frac{s_1}{2}} u_n|^2 \,dx+\frac{N(p-2)-4s_2}{2N(p-2)}\int_{\R^N}|(-\Delta)^{\frac{s_2}{2}} u_n|^2 \,dx.
\end{split}
\end{align}
Let us first treat the case $p=2+\frac{4s_1}{N}$. Suppose by contradiction that $\{\gamma(c_n)\} \subset \R$ is bounded.  Therefore, by arguing as the proof of Lemma \ref{coercive}, we are able to deduce that $\{u_n\} \subset H^{s_1}(\R^N)$ is bounded. Since $Q(u_n)=0$, by using \eqref{gn}, we obtain that
\begin{align*}
s_1\int_{\R^N}|(-\Delta)^{\frac{s_1}{2}} u_n|^2 \,dx +s_2\int_{\R^N}|(-\Delta)^{\frac{s_2}{2}} u_n|^2 \,dx
&=\frac{Ns_1}{N+2s_1} \int_{\R^N}|u_n|^{2+\frac{4s_1}{N}}\,dx \\
&\leq s_1\left(\frac{c_n}{c_{N, s_1}}\right)^{\frac{2s_1}{N}}
\int_{\R^N}|(-\Delta)^{\frac{s_1}{2}} u_n|^2 \,dx.
\end{align*}
Note that $c_n \to c_{N, s_1}^+$ as $n \to \infty$, it then follows that $\|(-\Delta)^{{s_2}/{2}} u_n\|_2=o_n(1)$. Therefore, from \eqref{ide111}, we find that $\gamma(c_n)=o_n(1)$. It is impossible, see Lemmas \ref{coercive} and \ref{nonincreasing}. Then we get that $\gamma(c) \to \infty$ as $c \to c_{N, s_1}$.

Next we  handle the case $p>2+\frac{4s_1}{N}$. In this case, by applying \eqref{gn}, we have that
\begin{align*}
s_1\int_{\R^N}|(-\Delta)^{\frac{s_1}{2}} u_n|^2 \,dx +s_2\int_{\R^N}|(-\Delta)^{\frac{s_2}{2}} u_n|^2 \,dx
&=\frac{N(p-2)}{2p}\int_{\R^N}|u_n|^p \,dx \\
&\leq C_{N,p,s_1}c^{\frac p2 -\frac{N(p-2)}{4s_1}}
\left(\int_{\R^N}|(-\Delta)^{\frac{s_1}{2}} u_n|^2 \,dx \right)^{\frac{N(p-2)}{4s_1}}.
\end{align*}
This readily infers that $\|(-\Delta)^{{s_1}/{2}} u_n\|_2\to \infty$ as $n \to \infty$. Utilizing \eqref{ide111}, we then get the desired result. Thus the proof is completed.
\end{proof}

\begin{lem} \label{limit2}
Let $N=1$, $2s_2 \geq 1$ and $p\geq 2+4s_1$ or $N \geq 1$, $2s_2<N$ and $2 + \frac{4s_2}{N}< p < \frac{2N}{N-2s_2}$. Then $\gamma(c) \to 0$ as $c \to \infty$.
\end{lem}
\begin{proof}
We first consider the case $p=2+\frac{4s_1}{N}$. Observe that
$$
E(w_t)=\frac{c}{2\|u\|_2^2} \left(1-\left(\frac{c}{c_{1,s_1}}\right)^{\frac{2s_1}{N}}\right) t^{2s_1}\int_{\R^N} |(-\Delta)^{\frac{s_1}{2}}u|^2 \, dx +\frac{c}{2\|u\|_2^2} t^{2s_2} \int_{\R^N} |(-\Delta)^{\frac{s_2}{2}}u|^2 \, dx,
$$
where $w$ is defined by \eqref{defw}. Then we have that
$$
0<\gamma(c) \leq \sup_{t>0} E(w_t)=\frac{s_1-s_2}{s_1}\left(\frac{s_2}{s_1}\right)^{\frac{s_2}{s_1-s_2}}\frac{\frac{c}{2\|u\|_2^2} \left( \int_{\R^N} |(-\Delta)^{\frac{s_2}{2}} u|^2 \,dx\right)^{\frac{s_1}{s_1-s_2}}}{\left(\left(\left(\frac{c}{c_{N,s_1}}\right)^{\frac{2s_1}{N}}-1\right) \int_{\R^N} |(-\Delta)^{\frac{s_1}{2}} u|^2 \,dx\right)^{\frac{s_2}{s_1-s_2}}}.
$$
This immediately yields that $\gamma(c) \to 0$ as $c \to \infty$, because of
$$
\frac{2s_1s_2}{N(s_1-s_2)}>1.
$$
Next we consider the case $p>2+\frac{4s_1}{N}$. For any $u \in S(1)$, we see that $c^{1/2} u \in S(c)$. From Lemma \ref{monotonicity}, we then know that there exists a constant $t_c>0$ such that $Q((c^{1/2}u)_{t_c})=0$. This means that
\begin{align*}
&t^{2(s_1-s_2)}_c s_1\int_{\R} |(-\Delta)^{\frac{s_1}{2}} u|^2 dx + s_2\int_{\R}|(-\Delta)^{\frac{s_2}{2}} u|^2dx=(c t_c^{2s_2})^{\frac {N(p-2)-4s_2}{4s_2}} c^{\frac p2-\frac{N(p-2)}{4s_2}}\frac{N(p-2)}{2p} \int_{\R}|u|^p dx.
\end{align*}
Therefore, we derive that $ct_c^{2s_2} \to 0$ as $c \to \infty$. Observe that
\begin{align*}
\gamma(c) \leq E((c^{1/2}u)_{t_c})&=E((c^{1/2}u)_{t_c})-\frac{2}{N(p-2)}Q((c^{1/2}u)_{t_c})\\
&=c t_c^{2s_1}\frac{N(p-2)-4s_1}{2N(p-2)} \int_{\R} |(-\Delta)^{\frac{s_1}{2}} u|^2 \,dx + c t_c^{2s_2}\frac{N(p-2)-4s_2}{2N(p-2)} \int_{\R} |(-\Delta)^{\frac{s_2}{2}} u|^2\, dx.
\end{align*}
Then we conclude that $\gamma(c) \to 0$ as $c \to \infty$. Thus the proof is completed.
\end{proof}

To further reveal some properties of the function $c \mapsto \gamma(c)$ as $c \to \infty$, we need to investigate the following zero mass equation,
\begin{align} \label{zfequ}
(-\Delta)^{s_1} u +(-\Delta)^{s_2} u =|u|^{p-2} u,
\end{align}
where $N>2s_2$ and $p\geq 2+\frac{4s_1}{N}$. The Sobolev space related to \eqref{zfequ} is defined by
$$
H:=\left\{u \in \dot{H}^{s_2}(\R^N) : \int_{\R^N}|(-\Delta)^{\frac{s_1}{2}} u|^2 \,dx <\infty\right\}
$$
equipped with norm
$$
\|u\|_H:=\left(\int_{\R^N}|(-\Delta)^{\frac{s_1}{2}} u|^2 \,dx\right)^{\frac 12}+\left(\int_{\R^N}|(-\Delta)^{\frac{s_2}{2}} u|^2 \,dx\right)^{\frac 12}, \quad  u \in H.
$$
It is standard to check that $H$ is a reflexive Banach space.

\begin{lem} \label{embedding}
Let $N>\max\{2s_1,\frac{2s_1s_2}{s_1-s_2}\}$, $0<s_2<s_1<1$ and $p \geq 2+\frac{4s_1}{N}$. Then there holds that
$$
\int_{\R^B}|u|^p \,dx \leq  C_{N, s_1,s_2} \left(\int_{\R^N}|(-\Delta)^{\frac{s_2}{2}} u|^2 \,dx\right)^{^{\frac{N(1-\theta)}{N-2s_2}}}\left(\int_{\R^N}|(-\Delta)^{\frac{s_1}{2}} u|^2 \,dx\right)^{^{\frac{N\theta}{N-2s_1}}},  \quad  u \in H,
$$
where $C_{N,s_1,s_2}>0$ is a constant and $0<\theta<1$ satisfying that
$$
p=\frac{2N(1-\theta)}{N-2s_2}+\frac{2N\theta}{N-2s_1}.
$$
\end{lem}

\begin{proof}
Since $N>2s_2$ and $N>\frac{2s_1s_2}{s_1-s_2}$, then
$$
2+\frac{4s_1}{N} >\frac{2N}{N-2s_2}.
$$
Due to $N>2s_1$ and $p<\frac{2N}{N-2s_1}$, then there exists a constant $0<\theta<1$ such that
$$
p=\frac{2N(1-\theta)}{N-2s_2}+\frac{2N\theta}{N-2s_1}.
$$
By using H\"older's inequality, we have that
$$
\int_{\R^N}|u|^p\,dx \leq \left(\int_{\R^N}|u|^{\frac{2N}{N-2s_2}} \,dx \right)^{1-\theta} \left(\int_{\R^N}|u|^{\frac{2N}{N-2s_1}} \,dx \right)^{\theta} .
$$
In view of Sobolev inequalities, we know that
$$
\int_{\R^N}|u|^{\frac{2N}{N-2s_2}} \,dx \leq C_{N, s_2}\left(\int_{\R^N}|(-\Delta)^{\frac{s_2}{2}} u|^2\,dx\right)^{\frac{N}{N-2s_2}}
$$
and
$$
\int_{\R^N}|u|^{\frac{2N}{N-2s_1}} \,dx \leq C_{N, s_1}\left(\int_{\R^N}|(-\Delta)^{\frac{s_1}{2}} u|^2\,dx\right)^{\frac{N}{N-2s_1}}.
$$
Therefore, the result of the lemma follows immediately and the proof is completed.
\end{proof}

Taking into account Lemma \ref{embedding}, we are able to seek for solutions to \eqref{zfequ} in $H$. Indeed, solutions to \eqref{zfequ} correspond to critical points of the following energy functional,
$$
E(u)=\frac 12 \int_{\R^N}|(-\Delta)^{\frac{s_1}{2}} u|^2 \, dx+\frac 12 \int_{\R^N}|(-\Delta)^{\frac{s_2}{2}} u|^2 \, dx-\frac 1p  \int_{\R^N} |u|^p \, dx.
$$

\begin{lem}
Let $N>\max\{2s_1,\frac{2s_1s_2}{s_1-s_2}\}$, $0<s_2<s_1<1$ and $p \geq 2+\frac{4s_1}{N}$. Then there exists a ground state solution to \eqref{zfequ} in $H$, namely the ground state energy $m$ is achieved, where
\begin{align} \label{defm}
m:=\left\{E(u) : u \in H\backslash\{0\}, E^\prime(u)=0\right\}.
\end{align}
\end{lem}

\begin{proof}
To prove this, it is equivalent to show that the functional $J$ restricted on $M$ admits a minimizer in $H$, where
$$
J(u):=\int_{\R^N}|(-\Delta)^{\frac{s_1}{2}} u|^2 \, dx+\int_{\R^N}|(-\Delta)^{\frac{s_2}{2}} u|^2 \, dx, \quad M:=\left\{u \in H: \int_{\R^N} |u|^p\,dx=1 \right\}.
$$
Let $\{u_n\} \subset M$ be a minimizing sequence. Without restriction, we may assume that $\{u_n\}$ is radially symmetric. Otherwise, we shall use its symmetric-decreasing rearrangement to replace $\{u_n\}$. Note that the embedding $H_{rad}\hookrightarrow L^p(\R^N)$ is compact for $N \geq 2$, where $H_{rad}$ denotes the subspace consisting of radially symmetric functions in $H$. Hence we easily obtain the existence of minimizers and the proof is completed.
\end{proof}

\begin{lem} \label{l2}
Let $N>\max\{2s_1+2,\frac{2s_1s_2}{s_1-s_2}\}$, $0<s_2<s_1<1$ and $p \geq 2+\frac{4s_1}{N}$. Then any solution to \eqref{zfequ} belongs to $L^2(\R^N)$.
\end{lem}
\begin{proof}
By using scaling techniques, it suffices to demonstrate that any solution to the equation
\begin{align} \label{zfequ1}
\left(\frac{k_{s_2}}{k_{s_1}}\right)^{\alpha s_1}(-\Delta)^{s_1} u +\left(\frac{k_{s_2}}{k_{s_1}}\right)^{\alpha s_2}(-\Delta)^{s_2} u =|u|^{p-2} u, \quad u \in H
\end{align}
belongs to $L^2(\R^N)$, where
$$
k_s:=2^{1-2s} \frac{\Gamma(1-s)} {\Gamma(s)} ,\quad \alpha=\frac{1}{s_1-s_2}.
$$
Indeed, if $u \in H$ is a solution to \eqref{zfequ}, then $\tilde{u} \in H$ is a solution to \eqref{zfequ1}, where $\tilde{u}$ is defined by
$$
\tilde{u}(x):=u\left({k_{s_1}^{\frac{\alpha}{2}}}/{k_{s_2}^{\frac{\alpha}{2}}}x\right), \quad x \in \R^N.
$$
For our purpose, by making use of the harmonic extension theory from \cite{CS}, we need to introduce the following extended problem
\begin{align}\label{fequ11}
\displaystyle
\left\{
\begin{aligned}
&-\mbox{div}(y^{1-2s_1} \nabla w+y^{1-2s_2} \nabla w)=0 \,\, \,  &\mbox{in} \,\, \R^{N+1}_+,\\
&-\frac{\partial w}{\partial {\nu}}=k_{s_1,s_2}|u|^{p-2}u \,\,\, &\mbox{on} \,\, \R^N \times \{0\},
\end{aligned}
\right.
\end{align}
where
$$
\frac{\partial w}{\partial {\nu}}:= \lim_{y \to 0^+} y^{1-2s_1} \frac{\partial w}{\partial y}(x, y)+y^{1-2s_2} \frac{\partial w}{\partial y}(x, y)  =- \frac{1}{k_{s_1}} (- \Delta )^{s_1} u(x)- \frac{1}{k_{s_2}} (- \Delta )^{s_2} u(x)
$$
and
$$
k_{s_1,s_2}:=\left(\frac{k_{s_1}^{s_2}}{k_{s_2}^{s_1}}\right)^{\alpha}.
$$
Let $\varphi \in C^{\infty}_0(\R^{N+1}_+ \backslash \mathcal{B}^+(0, R), \R)$ be a cut-off function such that $\varphi(x, y)=1$ for $|(x, y)| \geq 2R$, where
$$
\mathcal{B}^+(0, R):=\{(x, y) \in \R^{N+1}_+: |(x, y)| <R\}
$$
and $R>0$ is a constant to be determined later. For $R_1>2R$, we define $\psi:=\varphi h_{R_1}$, where $h_{R_1}$ is given by
\begin{align*}
h_{R_1}(x, y):=\left\{
\begin{aligned}
&|(x, y)|^{s_2}  &\quad \mbox{for} \,\, 2 R \leq |(x, y)|<R_1,\\
&R_1^{s_2} \left(1+\tanh((|(x, y)|-R_1)/R_1)\right) & \mbox{for} \,\, |(x, y)| \geq R_1.
\end{aligned}
\right.
\end{align*}
As a direct consequence, we see that
\begin{align} \label{cutoff}
\sup_{|(x, y)| \geq 2R} \frac{|(x,y)||\nabla \psi(x ,y)|}{\psi(x, y)}=1.
\end{align}
Multiplying \eqref{fequ11} by $\psi^2 w$ and integrating on $\R^{N+1}_+$, we have that
\begin{align*}
\int_{\R^{N+1}_+} \left(y^{1-2s_1} \nabla w +y^{1-2s_2} \nabla w \right) \cdot \nabla(w \psi^2) \, dx dy=k_{s_1,s_2}\int_{\R^N} |u|^p |\psi(\cdot, 0)|^2 \, dx.
\end{align*}
Observe that
$$
|\nabla(w \psi)|^2=\nabla w \cdot \nabla(w \psi^2) + |w|^2|\nabla \psi|^2.
$$
It then follows that
\begin{align} \label{estimate} \nonumber
\int_{\R^{N+1}_+} y^{1-2s_1} |\nabla (w \psi)|^2 +y^{1-2s_2} |\nabla (w \psi)|^2 \, dx &=\int_{\R^{N+1}_+} \left(y^{1-2s_1} |w|^2+ y^{1-2s_2} |w|^2 \right) |\nabla \psi|^2 \, dxdy \\
&\quad + \int_{\R^N} |u|^p |\psi(\cdot, 0)|^2 \, dx.
\end{align}
Applying H\"older's inequality, the definition of $\varphi$ and trace inequalities, we know that
\begin{align*}
\int_{\R^N} |u|^p |\psi(\cdot, 0)|^2 \, dx &\leq \left(\int_{\R^N} |u \psi(\cdot, 0)|^{p}\,dx \right)^{\frac2p} \left(\int_{\R^N \backslash B(0, R)} |u|^p \, dx\right)^{\frac{p-2}{p}}
\end{align*}
and
\begin{align*}
\int_{\R^N} |u \psi(\cdot, 0)|^{p}\,dx
&\leq \left(\int_{\R^N}|u \psi(\cdot, 0)|^{\frac{2N}{N-2s_1}}\,dx\right)^{\theta} \left(\int_{\R^N}|u \psi(\cdot, 0)|^{\frac{2N}{N-2s_2}}\,dx\right)^{1-\theta}\\
&\leq \widetilde{C}\left(\int_{\R^{N+1}_+} y^{1-2s_1} |\nabla (w\psi)|^2 \,dxdy\right)^{\frac{\theta N}{N-2s_1}} \left(\int_{\R^{N+1}_+} y^{1-2s_2} |\nabla (w\psi)|^2\, dxdy \right)^{\frac{(1-\theta)N}{N-2s_2}},
\end{align*}
where $0<\theta<1$ and
$$
p=\frac{2N\theta}{N-2s_1}+\frac{2N(1-\theta)}{N-2s_2}
$$
and $\widetilde{C}=C_{N,s_1,s_2}>0$ is a constant.
Define
$$
\delta(R):=\widetilde{C}k_{s_1,s_2}\left(\int_{\R^N \backslash B(0, R)} |u|^p \, dx\right)^{\frac{p-2}{p}},
$$
and note that $\delta(R) \to 0$ as $R \to \infty$. It then follows from \eqref{estimate} and Young's inequality that
\begin{align} \label{estimate1}
\begin{split}
&\left(1-\delta(R)\right)\int_{\R^{N+1}_+} y^{1-2s_1} |\nabla (w\psi)|^2 +y^{1-2s_2} |\nabla (w\psi)|^2 \, dx \\
&\leq \int_{\R^{N+1}_+} \left(y^{1-2s_1} |w|^2+ y^{1-2s_2} |w|^2 \right) |\nabla \psi|^2 \, dxdy.
\end{split}
\end{align}
Let us now estimate the terms in the right side hand of \eqref{estimate1}. Using \eqref{cutoff} and \cite[Theorem 1.1]{T}, we have that
\begin{align*}
\int_{\R^{N+1}_+} y^{1-2s_1}  |w|^2 |\nabla \psi|^2 dxdy
&\leq \int_{\mathcal{B}^+(0, 2R)} y^{1-2s_1} |w|^2 dxdy
+\int_{\R^{N+1}_+ \backslash \mathcal{B}^+(0, 2R)} y^{1-2s_1} \frac{|w \psi |^2}{|(x, y)|^{2}} dxdy \\
& \leq C_1(R) +\frac {4}{(b-2s_1-1)^2}\int_{\R^{N+1}_+} y^{1-2s_1} |\nabla (w\psi)|^2  dxdy,
\end{align*}
where $b>0$ is a constant satisfying $1+2s_1 < b < N+1$. Similarly, we can obtain that
\begin{align*}
\int_{\R^{N+1}_+} y^{1-2s_2} |w|^2 |\nabla \psi|^2  dxdy &\leq \int_{\mathcal{B}^+(0, 2R)} y^{1-2s_2} |w|^2 dxdy
+\int_{\R^{N+1}_+ \backslash \mathcal{B}^+(0, 2R)} y^{1-2s_2} \frac{|w\psi|^2}{|(x, y)|^{2}}dxdy \\
& \leq C_2(R) +\frac {4}{(b-2s_2-1)^2}\int_{\R^{N+1}_+} y^{1-2s_2} |\nabla (w\psi)|^2  dxdy,
\end{align*}
where $1+2s_2<1+2s_1 < b < N+1$. Choosing $b$ close to $N+1$ with $N>2s_1+2$ such that
$$
\frac {4}{(b-2s_1-1)^2}<1, \quad \frac {4}{(b-2s_2-1)^2}<1,
$$
and taking $R>0$ large enough, we then get from \eqref{estimate1} that
$$
\int_{\R^{N+1}_+} y^{1-2s_1} |\nabla (w\psi)|^2 +y^{1-2s_2} |\nabla (w\psi)|^2 \, dx  \leq C(R).
$$
Applying again \cite[Theorem 1.1]{T}, we then have that
$$
\int_{\R^N \backslash B(0, 2R)}\frac{|u \psi(\cdot, 0)|^2}{|x|^{2s_1}} +\frac{|u \psi(\cdot, 0)|^2}{|x|^{2s_2}}  \, dx \leq C
$$
uniformly with respect to $R_1$. Observe that
$$
\frac{|u \psi(\cdot, 0)|^2}{|x|^{2s_1}} +\frac{|u \psi(\cdot, 0)|^2}{|x|^{2s_2}} \to |u|^2 \quad \mbox{a.e.} \,\,\, |x| \geq 2R\,\,\,  \mbox{as}\,\, R_1 \to \infty.
$$
It then yields from Faton's Lemma that $u \in L^2(\R^N \backslash B(0, 2R))$. This indicates that $u \in L^2(\R^N)$ and the proof is completed.
\end{proof}

\begin{lem} \label{limit3}
Let $N>\max\{2s_1+2,\frac{2s_1s_2}{s_1-s_2}\}$, $0<s_2<s_1<1$ and $p \geq 2+\frac{4s_1}{N}$. Then there exists a constant $c_{\infty}>0$ such that $\gamma(c) = m$ for any $c \geq c_{\infty}$, where $m$ is defined  by \eqref{defm}.
\end{lem}

\begin{proof}
Note first that
$$
m=\left\{E(u) : u \in H \backslash\{0\}, Q(u)=0\right\}.
$$
Then we have that $\gamma(c) \geq m$ for any $c >c_0$. From Lemma \ref{l2}, we know that there exists $u \in H^{s_1}(\R^N)$ such that $E(u)=m$ and $Q(u)=0$. Consequently, we obtain that $\gamma(c_{\infty})=m$, where $c_{\infty}:=\|u\|_2^2$. Combining Lemma \ref{nonincreasing}, we then derive that $\gamma(c)=m$ for any $c \geq c_{\infty}$. Thus the proof is completed.
\end{proof}

From the lemmas above, we are now able to give the proof of Theorem \ref{thm4}.

\begin{proof}[Proof of Theorem \ref{thm4}]
The proof of Theorem \ref{thm4} follows directly from Lemmas \ref{lagrange}, \ref{nonincreasing}, \ref{decreasing}, \ref{continuous}, \ref{limit1}, \ref{limit2} and \ref{limit3}.
\end{proof}

\section{Local well-posedness of solutions to the Cauchy problem} \label{section6}

In this section, we shall demonstrate Theorem \ref{pb wellposedness}. First we need to present the definition of admissible pairs. For convenience, in the remaining sections, we shall replace the notation $\psi(t)$ by $u(t)$ to denote solutions to \eqref{evolv pb0}.


\begin{defi}
Let $N\geq 2$ and $s\in (0,1]$. Any pair $(q,r)$ of positive real numbers is said to be $s$-admissible, if
$q,r\geq 2$ and
$$\frac{2s}{q}+\frac{N}{r}=\frac{N}{2}.$$
Such a set of $s$-admissible pairs is denoted by $\Gamma_s$.
\end{defi}

The first main result in this section is with respect to a family of Strichartz estimates without loss of regularity, which are useful to control radially symmetric solutions of \eqref{evolv pb0}.

\begin{thm}\label{strichartz th}
Let $N\geq 2$, $\frac 12 <s_2<s_1< 1$, $u, u_0$ and $F$ are radially symmetric in space and satisfy the equation
\begin{align}\label{evolv pb}
\left\{
\begin{aligned}
&i\partial_tu-(-\Delta)^{s_1}u-(-\Delta)^{s_2}u=F(t,x), \\
&u(0,x)=u_0(x), \quad x\in \R^N.
\end{aligned}
\right.
\end{align}
Then
\begin{align}\label{Radial Strichartz}
\|u\|_{L_t^q L^r_x} \lesssim \|u_0\|_2+\|F\|_{L_t^{\tilde{q}^{\prime}}L_x^{\tilde{r}^{\prime}}},
\end{align}
if $(q,r)$ and $(\tilde{q},\tilde{r})$ belong to $\Gamma_{s_1}\cup \Gamma_{s_2}$ and either
$(q,r)\neq (2,\infty)$ or $(\tilde{q}^{\prime},\tilde{r}^{\prime})\neq (2,\infty)$.
\end{thm}

To prove Theorem \ref{strichartz th}, we first need to introduce some notations and preliminary results. Let us introduce a nonnegative smooth even function $\varphi:\R^N\to [0,1]$ such that $\mbox{supp} \, \varphi \subset \{x\in \R^N : |x|\leq 2\}$ and $\varphi(x)=1$ if $|x|\leq 1$.
Let $\psi (x):=\varphi(x)-\varphi(2x)$ and $P_k$ be the Littelwood-Paley projector for $k\in \Z$, namely
$$
P_kf:= \mathcal{F}^{-1}\psi(2^{-k}|\xi|)\mathcal{F}f.
$$
Recall that
$$
\|f\|_p^2\sim \sum_{k\in\Z}\|P_kf\|_p^2\,\,\, \text{and} \,\,\, \|f\|_{H^s}^2\sim \sum_{k\in\Z}2^{2s}\|P_kf\|_2^2.
$$
It is clear to see that the function $\phi(r):=r^{2s_1}+r^{2s_2}$ for any $r\in \R_+$ satisfies the following conditions introduced in \cite{GuPe08,GuWa14}.

(H1) There exists $m_1 > 0$ such that for any $\alpha \geq2$ and $\alpha \in \N$,
$$
|\phi^{\prime}(r)|\sim r^{m_1-1}\,\,\, \text{ and }\,\,\, |\phi^{(\alpha)}(r)|\lesssim r^{m_1-\alpha}, \quad r\geq 1.
$$

(H2) There exists $m_2 > 0$ such that for any $\alpha \geq2$ and $\alpha \in \N$,
$$
|\phi^{\prime}(r)|\sim r^{m_2-1}\,\,\, \text{ and }\,\,\, |\phi^{(\alpha)}(r)|\lesssim r^{m_2-\alpha}, \quad 0<r< 1.
$$

(H3) There exists $\alpha_1>0$ such that
$$
|\phi^{\prime\prime}(r)|\sim r^{\alpha_1-2}, \quad r\geq 1.
$$

(H3) There exists $\alpha_2>0$ such that
$$
|\phi^{\prime\prime}(r)|\sim r^{\alpha_2-2}, \quad 0<r< 1.
$$
Precisely, for our case one has $m_1=\alpha_1=2s_1$ and $m_2=\alpha_2=2s_2$. Let us now denote
\begin{align*}
m(k)=\alpha(k):=
\left\{
\begin{aligned}
&2s_1, &\text{ for }k\geq 0,\\
&2s_2, &\text{ for }k< 0.
\end{aligned}
\right.
\end{align*}
Thus, according to \cite[Theorem 1.2]{GuWa14}, the following result holds.

\begin{lem}\label{propos}
Suppose $N\geq 2, k\in \Z$, $(4N+2)/(2N-1)\leq q\leq +\infty$ and $u_0 \in L^2(\R^N)$ is radially symmetric. Then
$$
\|S(t)P_ku_0\|_{L_{t,x}^q(\R^{N+1})} \lesssim 2^{\left(\frac{N}{2}-\frac{N+m(k)}{2}\right)k}\|u_0\|_2,
$$
where $S(t)$ denotes the evolution group related to \eqref{evolv pb}, namely
$$
S(t)u_0:=\mathcal{F}^{-1}(e^{-it(|\xi|^{2s_1}+|\xi|^{2s_2})}\mathcal{F}u_0).
$$
\end{lem}


\begin{defi}
Let $N\geq 2$. The exponent pair $(q,r)$ is said to be $N$-D radial Schr\"{o}dinger-admissible, if $q,r\geq 2$ and
$$
\frac{4N+2}{2N-1}\leq q\leq \infty, \quad \frac{2}{q}+\frac{2N-1}{r}\leq N-\frac{1}{2}
$$
or
$$
2\leq q<\frac{4N+2}{2N-1},\quad \frac{2}{q}+\frac{2N-1}{r}\leq N-\frac{1}{2}.
$$
\end{defi}

In a similar way as the proof of \cite[theorem 1.5]{GuWa14}, from Lemma \ref{propos}, one can derive the following interesting result.

\begin{lem} \label{lem6.1}
Let $N\geq 2, k\in \Z, \frac{1}{2}<s_2<s_1<1$ and let
\begin{align}\label{Csq}
C_{s_1,s_2}^{q,r}(k):=
\left\{
\begin{aligned}
&2^{k\left(\frac{N}{2}-\frac{2s_1}{q}-\frac{N}{r}\right)}, &\text{ for }k\geq 0,\\
&2^{k\left(\frac{N}{2}-\frac{2s_2}{q}-\frac{N}{r}\right)}, &\text{ for }k< 0.
\end{aligned}
\right.
\end{align}
Then for any radial function $u_0\in L^2(\R^N)$, there holds that
\begin{align}\label{strichartz pk}
\|S(t)P_ku_0\|_{L_t^q L^r_x(\R \times \R^N)} \lesssim C_{s_1,s_2}^{q,r}(k)\|u_0\|_2,
\end{align}
if $(q,r)$ is $N$-D radial Schr\"{o}dinger-admissible.
\end{lem}


\begin{lem}\label{lemme1}\cite[Lemma 3.2]{GuWa14}
Assume $1\leq q,r \leq \infty$, $\frac{1}{q}+\frac{1}{q^{\prime}}=\frac{1}{r}+\frac{1}{r^{\prime}}=1$ and $k\in \Z$. If for any $u_0\in L_{rad}^2(\R^N)$ there exists $C(k)>0$ such that
$$
\|S(t)P_ku_0\|_{L_t^qL^r_x}\lesssim C(k)\|u_0\|_2,
$$
then for any $f\in L_t^{q^\prime}L^{r^\prime}_x$ and $f$ is radially symmetric in space, there holds that
$$
\left\|\int_{\R}S(-\tau)P_kf(\tau,\cdot)\,d\tau\right\|_2 \lesssim C(k)\|f\|_{L_t^{q^\prime}L^{r^\prime}_x}.
$$
\end{lem}

\begin{lem}\label{lemma2}(Christ-Kiselev \cite{CK})
Assume $1\leq p_1,q_1,p_2,q_2\leq \infty$ with $p_1>p_2$. If for any $f\in L_t^{p_2}L^{q_2}_x$ radially symmetric in space, there exists $C(k)>0$ such that
$$
\left\|\int_{\R}S(t-\tau)P_kf(\tau,\cdot)d\tau\right\|_{L_t^{p_1}L^{q_1}_x}\lesssim C(k) \|f\|_{L_t^{p_2}L^{q_2}_x},
$$
then
$$
\left\|\int_0^tS(t-\tau)P_kf(\tau,\cdot)d\tau\right\|_{L_t^{p_1}L^{q_1}_x}\lesssim C(k) \|f\|_{L_t^{p_2}L^{q_2}_x}
$$
holds with the same bound $C(k)$ and for any $f\in L_t^{p_1}L^{q_1}_x$ radially symmetric in space.
\end{lem}

Now, based on Lemmas \ref{lem6.1} \ref{lemme1} and \ref{lem6.4}, we are able to prove the following family of Strichartz estimates of solutions to \eqref{evolv pb}.

\begin{lem} \label{lem6.4}
Let $N\geq 2$, $\frac 12<s_2<s_1< 1$, $u, u_0$ and $F$ are radially symmetric in space and satisfy \eqref{evolv pb}. Then
\begin{align}\label{Strichartz estimate}
\|u\|_{L_t^q L^r_x} \lesssim \left\|C_{s_1,s_2}^{q,r}(k)\left\|P_k^2u_0\right\|_2\right\|_{l_k^2}+\left\|C_{s_1,s_2}^{\tilde{q},\tilde{r}}(k)\|P_kF\|_{L_t^{\tilde{q}^{\prime}}L^{\tilde{r}^{\prime}}_x}\right\|_{l_k^2},
\end{align}
if $(q,r)$ and $(\tilde{q}^{\prime},\tilde{r}^{\prime})$ are $N$-D radial Schr\"{o}dinger-admissible pairs, either $(\tilde{q}^{\prime},\tilde{r}^{\prime})\neq (2,\infty)$ or $(q,r)\neq (2,\infty)$, where $C_{s_1,s_2}^{q,r}>0$ is defined by \eqref{Csq}.
\end{lem}
\begin{proof}
In view of Duhamel's principle, we first have that
$$
u=S(t)u_0-i\int_0^tS(t-\tau)F(\tau,\cdot)\,d\tau.
$$
Let $P_k^2$ be the Littelwood-Paley projector associated to $\psi^2$. Hence
\begin{align*}
P_k^2f = \mathcal{F}^{-1} \psi^2(2^{-k}|\xi|)\mathcal{F}f= \mathcal{F}^{-1} \psi(2^{-k}|\xi|)\psi(2^{-k}|\xi|)\mathcal{F}f = \mathcal{F}^{-1} \psi(2^{-k}|\xi|)\mathcal{F}P_kf= P_kP_kf.
\end{align*}
Similarly, one gets that $P_k^3=P_kP_k^2$. Note that
$$
P_k^3u=S(t)P_k^3u_0-i\int_0^tS(t-\tau)P_k^3F(\tau,\cdot)\,d\tau.
$$
It then yields that
\begin{align}\label{ineq}
\left\|P_k^3u\right\|_{L_t^q L^r_x} \lesssim \left\|S(t)P_k^3u_0\right\|_{L_t^q L^r_x}+\left\|\int_0^tS(t-\tau)P_k^3F(\tau,\cdot)\,d\tau\right\|_{L_t^qL^r_x}.
\end{align}
According to Lemmas \ref{lem6.1} and \ref{lemme1}, we have that
$$
\left\|\int_{\R}S(-\tau)P_k^2F(\tau,\cdot)\,d\tau\right\|_2\lesssim C_{s_1,s_2}^{\tilde{q},\tilde{r}}(k)\|P_kF\|_{L_t^{\tilde{q}^\prime}L^{\tilde{r}^\prime}_x}.
$$
Therefore, from Lemma \ref{lem6.1}, we derive that
\begin{align*}
\left\|\int_{\R}S(t-\tau)P_k^3 F(\tau,\cdot) \,d\tau\right\|_{L_t^q L^r_x}
& = \left\|S(t)P_k\int_{\R}S(-\tau)P_k^2 F(\tau,\cdot) \, d\tau\right\|_{L_t^q L^r_x}\\
& \lesssim \left \|\int_{\R}S(-\tau)P_k^2 F(\tau,\cdot) \, d\tau \right \|_2 \lesssim C_{s_1,s_2}^{\tilde{q},\tilde{r}}(k)\|P_kF\|_{L_t^{\tilde{q}^\prime}L^{\tilde{r}^\prime}_x}.
\end{align*}
It then follows from Lemma \ref{lemma2} that
\begin{align}\label{ineq0}
\left\|\int_0^t S(t-\tau)P_k^3 F(\tau,\cdot) \,d\tau\right\|_{L_t^qL^r_x}\lesssim C_{s_1,s_2}^{\tilde{q},\tilde{r}}(k)\|P_kF\|_{L_t^{\tilde{q}^\prime}L^{\tilde{r}^\prime}_x}.
\end{align}
Coming back to \eqref{ineq} and using Lemma \ref{lem6.1} and \eqref{ineq0} results in
$$
\left\|P_k^3u\right\|_{L_t^q L^r_x} \lesssim C_{s_1,s_2}^{q,r}(k)\left\|P_k^2u_0\right\|_2+C_{s_1,s_2}^{\tilde{q},\tilde{r}}(k)\|P_kF\|_{L_t^{\tilde{q}^\prime} L^{\tilde{r}^\prime}_x},
$$
from which we then conclude that
$$
\|u\|_{L_t^q L^r_x} \lesssim \left\|\left\|P_k^3u\right\|_{L_t^q L^r_x}\right\|_{l_2^k}\lesssim \left\|C_{s_1,s_2}^{q,r}(k)\|P_k^2u_0\|_2\right\|_{l_2^k}+\left\|C_{s_1,s_2}^{\tilde{q},\tilde{r}}(k)\|P_kF\|_{L_t^{\tilde{q}^\prime}L^{\tilde{r}^\prime}_x}\right\|_{l_2^k}.
$$
Thus the proof is completed.
\end{proof}

We are now ready to prove Theorem \ref{strichartz th}.

\begin{proof}[Proof of Theorem \ref{strichartz th}]
If $(q,r)\in \Gamma_{s_2}$, then \eqref{Csq} becomes
$$
C_{s_1,s_2}^{q,r}(k):=
\left\{
\begin{aligned}
&2^{\frac{2k}{q}(s_2-s_1)}, &\text{ for }k\geq 0,\\
&1, &\text{ for }k< 0.
\end{aligned}
\right.
$$
It then leads to
\begin{align*}
\left\|C_{s_1,s_2}^{q,r}(k)\left\|P_k^2u_0\right\|_2\right\|_{l_2^k}^2
  & \lesssim \sum_{k<0} \left\|P_k^2u_0\right\|_2^2+ \sum_{k\geq0} 2^{\frac{4k}{q}(s_2-s_1)}\left\|P_k^2u_0\right\|_2^2\\
  & \lesssim \sum_{k\in \Z} \left\|P_k^2u_0\right\|_2^2 \sim \|u_0\|_2^2.
\end{align*}
In the case when $(q,r)\in \Gamma_{s_1}$, one obtains that
$$
C_{s_1,s_2}^{q,r}(k):=
\left\{
\begin{aligned}
&1, &\text{ for }k\geq 0,\\
&2^{\frac{2k}{q}(s_1-s_2)}, &\text{ for }k< 0.
\end{aligned}
\right.
$$
As a consequence, there holds that
\begin{align*}
\left\|C_{s_1,s_2}^{q,r}(k)\left\|P_k^2u_0\right\|_2\right\|_{l_2^k}^2
&\lesssim \sum_{k<0} 2^{\frac{4k}{q}(s_1-s_2)}\left\|P_k^2u_0\right\|_2^2+\sum_{k\geq0}\left\|P_k^2u_0\right\|_2^2\\
&\lesssim \sum_{k\in \Z} \left\|P_k^2u_0\right\|_2^2\sim \|u_0\|_2^2.
\end{align*}
Similarly, whenever $(\tilde{q},\tilde{r})$ belongs to $\Gamma_{s_1}\cup \Gamma_{s_2}$, one also gets that
$$
\left\|C_{s_1,s_2}^{\tilde{q},\tilde{r}}(k)\|P_kF\right\|_{L_t^{\tilde{q}^\prime}L^{\tilde{r}^\prime}_x}\|_{l_2^k}^2\lesssim \|\|P_kF\|_{L_t^{\tilde{q}^\prime}L^{\tilde{r}^\prime}_x}\|_{l_2^k}^2\lesssim \|F\|_{L_t^{\tilde{q}^\prime}L^{\tilde{r}^\prime}_x}^2.
$$
Making use of Lemma \ref{lem6.4}, we then have the desired conclusion. Thus the proof is completed.
\end{proof}

Let us now present chain rules for fractional Laplacian adapted to prove Theorem \ref{pb wellposedness}.

\begin{lem}\label{chain rules}
Let $s\in (0,1]$ and $1<p, p_i, q_i<\infty$ satisfying $\frac{1}{p}=\frac{1}{p_i}+\frac{1}{q_i}$ and i=1,2.
\begin{itemize}
\item [$(\textnormal{i})$] There holds that
$$
\|(-\Delta)^{\frac{s}{2}}(uv)\|_{p}\lesssim \|(-\Delta)^{\frac{s}{2}}u\|_{p_1}\|v\|_{q_1}+\|u\|_{p_2}\|(-\Delta)^{\frac{s}{2}}v\|_{q_2}.
$$
\item [$(\textnormal{i})$] If $G\in C^1(\C)$, then
$$
\|(-\Delta)^{\frac{s}{2}}G(u)\|_{p}\lesssim \|G^\prime(u)\|_{p_1}\|(-\Delta)^{\frac{s}{2}}u\|_{q_1}.
$$
\end{itemize}
\end{lem}

\begin{proof}[Proof of Theorem \ref{pb wellposedness}]
To prove Theorem \ref{pb wellposedness}, we shall employ the contraction mapping principle. Let us first introduce some notations. Denote
$$
(q_j,r):=\left(\frac{4s_jp}{N(p-2)},p\right) \in \Gamma_{s_j}, \quad j=1,2.
$$
For $T,R>0$, we define
$$
Y_T:=C_T(H_{rad}^{s_1}(\R^N))\cap L_T^{q_1}(W^{s_1,r}(\R^N)) \,\,\, \text{   and    }\,\,\, B_T(R):=\{u\in Y_T : \|u\|_{T}\leq R\},
$$
where
$$
\|u\|_{T}:=\|u\|_{L_T^\infty H^{s_1}\cap L_T^{q_1} W^{s_1,r}}.
$$
The closed ball $B_T(R)$ is equipped with the complete distance
$$
d(u,v):=\|u-v\|_{L_T^\infty L^2_x\cap L_T^{q_1} L^r_x}.
$$
Given $u_0\in H_{rad}^{s_1}(\R^N)$, we define a mapping by
$$
\Phi(u)(t):=S(t)u_0+i\int_0^tS(t-s)|u(\tau)|^{p-2}u(\tau)\,d\tau.
$$
In the following, we are going to prove the existence of $T>0$ sufficiently small such that $\Phi$ defines a contraction mapping on $B_T(R)$. For any $u,v \in B_T(R)$, applying Strichartz estimate \eqref{Radial Strichartz}, one has that
$$
d(\Phi(u),\Phi(v))\lesssim \||u|^{p-2}u-|v|^{p-2}v\|_{L_T^{q_1^\prime}L^{r^\prime}_x}
$$
The mean value theorem gives that
$$
||u|^{p-2}u-|v|^{p-2}v|\lesssim (|u|^{p-2}+|v|^{p-2})|u-v|.
$$
From H\"{o}lder's inequality and the Sobolev embedding $H^{s_1}(\R^N)\hookrightarrow L^r(\R^N)$ for any $2 \leq r \leq \frac{2N}{N-2s_s}$, we then obtain that
\begin{align*}
d(\Phi(u),\Phi(v)) & \lesssim \||u|^{p-2}+|v|^{p-2}\|_{L_T^{\frac{2s_1p}{2s_1p-N(p-2)}}L^{\frac{p}{p-2}}_x}\|u-v\|_{L_T^{q_1}L^r_x}\\
& \lesssim T^{\frac{2s_1p-N(p-2)}{2s_1p}}\||u|^{p-2}+|v|^{p-2}\|_{L_T^{\infty}L^{\frac{p}{p-2}}_x}\|u-v\|_{L_T^{q_1}L^r_x}\\
& \lesssim T^{\frac{2s_1p-N(p-2)}{2s_1p}}\left(\|u\|_{L_T^\infty L^p_x}^{p-2}+\|v\|_{L_T^\infty L^p_x}^{p-2}\right)\|u-v\|_{L_T^{q_1}L^r_x}\\
& \lesssim T^{\frac{2s_1p-N(p-2)}{2s_1p}}\left(\|u\|_{L_T^\infty H^{s_1}}^{p-2}+\|v\|_{L_T^\infty H^{s_1}}^{p-2}\right)\|u-v\|_{L_T^{q_1}L^r_x}.
\end{align*}
This infers that
\begin{align}\label{phi contract}
d(\Phi(u),\Phi(v))\lesssim T^{\frac{2s_1p-N(p-2)}{2s_1p}}R^{p-2}d(u,v).
\end{align}
Note that the condition $2<p<\frac{2N}{N-2s_1}$ implies that $2s_1p-N(p-2)>0$. Next suppose
$\|S(\cdot)u_0\|_{T}<\frac{\widetilde{C}R}{2}$ and denote
$$
\theta :=\frac{2s_1p}{2s_1p-N(p-2)},
$$
where $\widetilde{C}>0$ is a small constant determined later. Taking $v=0$ and $T>0$ small enough in \eqref{phi contract}, one derives that
\begin{align*}
\|\Phi(u)\|_{L_T^\infty L^2_x\cap L_T^{q_1} L^r_x}
& \lesssim \frac{\widetilde{C}R}{2}+ T^{\frac{2s_1p-N(p-2)}{2s_1p}}R^{p-1}.
\end{align*}
Moreover, using Strichartz estimate \eqref{Radial Strichartz}, the chain rules, see Lemma \ref{chain rules}, H\"{o}lder's inequality and
Sobolev embedding, one gets that
\begin{align*}
\|\Phi(u)\|_{L_T^\infty \dot{H}^{s_1} \cap L_T^{q_1} \dot{W}^{s_1,r}}
& \lesssim \|S(\cdot)u_0\|_T+\|(-\Delta)^{\frac{s_1}{2}}(|u|^{p-2}u)\|_{L_T^{q_1^\prime} L^{r^\prime}_x}\\
& \lesssim \frac{\widetilde{C}R}{2}+ \|(-\Delta)^{\frac{s_1}{2}} u\|_{L_T^{q_1} L^{r}_x}\|u\|_{L_T^{\theta} L^{r}_x}^{p-2}\\
& \lesssim \frac{\widetilde{C}R}{2}+ \| u\|_{L_T^{q_1} W^{s_1,r}}\|u\|_{L_T^{\theta} H^{s_1}}^{p-2}\\
& \lesssim \frac{\widetilde{C}R}{2}+ T^{\frac{p-2}{\theta}}\| u\|_{L_T^{q_1} W^{s_1,r}}\|u\|_{L_T^\infty H^{s_1}}^{p-2}\\
& \lesssim \frac{\widetilde{C}R}{2}+ \frac{\widetilde{C}^{p-1}T^{\frac{p-2}{\theta}}R^{p-1}}{2^{p-1}}.
\end{align*}
In conclusion, by taking $\widetilde{C}>0$ small enough, we obtain that $\Phi$ is a contraction mapping on $B_T(R)$ for some $T>0$ small enough. This then leads to the local existence of solutions to \eqref{evolv pb0}. Uniqueness of maximal solutions to \eqref{evolv pb0} follows from \eqref{phi contract} for small time. Then, by using standard translation argument, one obtains uniqueness of solutions for all existing time.

Now, we focus on our attention to prove the conservation laws. Let $u\in C_{T}(H^{s_1}(\R^N))$ be a maximal solution to the evolving problem \eqref{evolv pb0}. Since $u(t)\in H^{s_1}(\R^N)$, then we can multiply the equation in \eqref{evolv pb0} by $i\bar{u}(t)$ and integrate over $\R^N$ to find that
$$
\frac{d}{dt}\|u(t)\|_2^2=0,\quad \forall \,\, t\in [0,T).
$$
This implies the conservation of mass.
To see the conservation of energy, let us introduce the operator
$$
J_{\varepsilon}:=(I+\varepsilon(-\Delta)^{s_1}+\varepsilon(-\Delta)^{s_2})^{-1},\quad \varepsilon>0.
$$
Using standard functional arguments in\cite{Caz89}, one can verify that
\begin{itemize}
\item [$(\textnormal{i})$] $J_{\varepsilon}$ defines a bounded mapping from $H^{-s_1}(\R^N)$ into $H^{s_1}(\R^N)$.
\item [$(\textnormal{ii})$] If $f\in L^p(\R^N)$, then $J_{\varepsilon}f\in L^p(\R^N)$ and $\|J_{\varepsilon}f\|_{L^p}\lesssim \|f\|_{L^p}$ for some $p\in [1,\infty)$.
\item [$(\textnormal{iii})$] If $X$ is either of the space $H^{s_1}(\R^N), L^2(\R^N), H^{-s_1}(\R^N)$, then for every $f\in X$, there holds that
$J_{\varepsilon}f\rightarrow f$ in $X$ as $\varepsilon \to 0^+$.
\end{itemize}
At this point, the energy conservation follows easily by adaptation of the arguments for Strichartz solutions developed in \cite{Oz06}.

Let us now prove the third assertion. Define
$$
X(t):=\|(-\Delta)^{\frac{s_1}{2}}u(t)\|_2^2+\|(-\Delta)^{\frac{s_2}{2}}u(t)\|_2^2, \quad t \in [0, T).
$$
Taking into account of the conservation laws and applying Gagliardi-Nirenberg inequality \eqref{gn}, we have that
\begin{align*}
E(u_0)=E(u(t))
& \geq \frac{1}{2}X(t)-\frac{C_{N,p,s_1}}{p}\|u_0\|_2^{p-\frac{N(p-2)}{2s_1}}
\|(-\Delta)^{\frac{s_1}{2}}u(t)\|_{2}^{\frac{N(p-2)}{2s_1}}\\
& \geq X(t)\left(\frac{1}{2}-\frac{C_{N,p,s_1}}{p}\|u_0\|_2^{p-\frac{N(p-2)}{2s_1}}
X(t)^{\frac{N(p-2)-4s_1}{4s_1}}\right).
\end{align*}
It then follows that $ \sup_{[0,T)}X(t)<\infty$ if $p<2+\frac{4s_1}{N}$ or $p=2+\frac{4s_1}{N}$ and
$$
\|u_0\|_2<\left(\frac{N+2s_1}{NC_{N,s_1}}\right)^{\frac{N}{4s_1}}.
$$
This completes the proof.
\end{proof}

\section{Blowup versus global existence of solutions to the Cauchy problem} \label{section7}

In this section, our aim is to prove Theorem \ref{blow-up vs global solt}, namely we shall derive general criteria with respect to the existence of global/non-global solutions to \eqref{evolv pb0}. For this, we need to introduce at first virial type inequality in the spirit of the recent work \cite{BoHiLe2016}.
Let us first introduce $\chi\in C_0^\infty(\R^N, \R^+)$ as a radial cut-off function satisfying
\begin{align}
\chi(r)=\chi(|x|):=
\left\{
\begin{aligned}
&\frac{1}{2}|x|^2,&\text{ for } |x|\leq 1, \\
&C,&\text{ for } |x|\geq 10,
\end{aligned}
\right.
\quad \text{ and }\,\,\,\chi^{\prime\prime}(r)\leq 1, \quad \forall \,\, r \geq 0.
\end{align}
For $R>0$, we define
$$
\chi_R:=R^2\chi\left(\frac{\cdot}{R}\right).
$$
It is simple to check that $\chi_R$ satisfies the following properties,
\begin{align}\label{bound psi}
\chi_R^{\prime\prime}(r)\leq 1, \quad \chi_R^{\prime}(r)\leq r, \quad \Delta \chi_R(r)\leq N, \quad r \geq 0.
\end{align}
The localized virial type quantity is defined by
$$
M_{\chi_R}[u]:=2 Im \int_{\R^N}\overline{u}\nabla\chi_R \cdot \nabla u \,dx.
$$

\begin{lem}\label{Virial}
Let $N\geq 2$, $\frac1 2<s_2<s_1<1$ and $2<p<\frac{2N}{N-2s_1}$. Assume that $u\in C_T(H_{rad}^{s_1}(\R^N))$ is a solution
to \eqref{evolv pb0}.
\begin{enumerate}
\item [$(\textnormal{i})$] For every $R>0$ and $\varepsilon>0$ small enough, then there holds that
\begin{align*}
\frac{d}{dt}M_{\chi_R}[u]
& \leq 2N(p-2)E(u_0)-(N(p-2)-4s_1)\|(-\Delta)^{\frac{s_1}{2}}u\|_{2}^2-(N(p-2)-4s_2)\|(-\Delta)^{\frac{s_2}{2}}u\|_{2}^2\\
& \quad +C\left(R^{-2s_2}+R^{-\frac{(p-2)(N-1)}{2}+\varepsilon s_1}\|(-\Delta)^{\frac{s_1}{2}}u\|_{2}^{\frac{p-2}{2s_1}+\varepsilon}\right).
\end{align*}
\item [$(\textnormal{ii})$] If $p=2+\frac{4s_1}{N}$ and $E(u_0)<0$ , then for some $R$ sufficiently large, there holds that
$$
\frac{d}{dt}M_{\chi_R}[u] < 4s_1E(u_0).
$$
\end{enumerate}
\end{lem}
\begin{proof}
The proof here is an adaptation of the one mentioned in \cite{BoHiLe2016}. To begin with, let us first introduce a self-adjoint differential operator
$$
\Gamma_\chi:=-i(\nabla \cdot \nabla \chi +\nabla \chi \cdot\nabla).
$$
It acts on a function $f$ as follows,
$$
\Gamma_\chi f=-i(\nabla\cdot ((\nabla\chi)f)+\nabla\chi \cdot \nabla f).
$$
One can check that
$$
M_{\chi_R}[u(t)]=\langle u(t),\Gamma_{\chi_R} u(t)\rangle.
$$
For $m>0$, we also introduce the function
$$
u_m:=\sqrt{\frac{\sin(\pi s)}{\pi}}\frac{1}{m-\triangle}u=\sqrt{\frac{\sin(\pi s)}{\pi}}\mathcal{F}^{-1}(\frac{\mathcal{F}u}{|\cdot|^2+m}).
$$
If $[X,Y]:=XY-YX$ denotes the commutator of $X$ and $Y$, then, by taking the time derivative and using \eqref{evolv pb0}, one gets that
\begin{align*}
\frac{d}{dt}M_{\chi_R}[u(t)]
=\langle u(t), [(-\Delta)^{s_1}+(-\Delta)^{s_2},i\Gamma_{\chi_R}]u(t)\rangle
+ \langle u(t), [-|u|^{p-2}u,i\Gamma_{\chi_R}]u(t)\rangle.
\end{align*}
According to computations developed in \cite{BoHiLe2016}, one has that
\begin{align*}
\frac{d}{dt}M_{\chi_R}[u(t)]
& \leq 4s_1\|(-\Delta)^{\frac{s_1}{2}}u\|_{2}^2+4s_2\|(-\Delta)^{\frac{s_2}{2}}u\|_{2}^2
- \frac{2(p-2)}{p}\int_{\R^N}|u|^p\Delta \psi_R \,dx+C\left(R^{-2s_1}+R^{-2s_2}\right)\\
& \leq 4s_1\|(-\Delta)^{\frac{s_1}{2}}u\|_{2}^2+4s_2\|(-\Delta)^{\frac{s_2}{2}}u\|_{2}^2- \frac{2N(p-2)}{p}\int_{\R^N}|u|^p\,dx\\
& \quad +C\left(R^{-2s_1}+R^{-2s_2}+R^{-\frac{(p-2)(N-1)}{2}+\varepsilon s_1}\|(-\Delta)^{\frac{s_1}{2}}u\|_{2}^{\frac{p-2}{2s_1}+\varepsilon}\right),
\end{align*}
for any $0<\varepsilon <\frac{(2s_1-1)(p-2)}{2s_1}$  and some constant $C:=C(\|u_0\|_2,N,\varepsilon,s_1,p)>0$. Hence, by
conservation of energy, one obtains the virial type inequality in the energy subcritical case.

Now suppose $p=2+\frac{4s_1}{N}$ and denote $\chi_1:=1-\chi_R^{\prime\prime}$ and $\chi_2:=N-\Delta \chi_R(r)$. Recall that $\chi_1$ and $\chi_2$ are nonnegative by \eqref{bound psi}. Using similar computations from \cite{BoHiLe2016}, we derive that
\begin{align*}
\frac{d}{dt}M_{\chi_R}[u(t)]
& \leq 4s_1\|(-\Delta)^{\frac{s_1}{2}}u\|_{2}^2+4s_2\|(-\Delta)^{\frac{s_2}{2}}u\|_{2}^2 -4\int_{0}^{\infty}m^{s_1}\int_{\R^N}\chi_1|\nabla u_m|^2\,dxdm\\
& \quad -\frac{4s_1N}{N+2s_1}\int_{\R^N}|u|^p\,dx+\frac{4s_1N}{N+2s_1}\int_{|x|\geq R}\chi_2|u|^p\,dx +O(R^{-2s_1}+R^{-2s_1})\\
& \leq 8s_1E(u_0)-4\int_{0}^{\infty}m^{s_1}\int_{\R^N}\chi_1|\nabla u_m|^2\,dxdm +\frac{4s_1N}{N+2s_1}\int_{\R^N}\chi_2|u|^p\,dx +O(R^{-2s_1}).
\end{align*}
Estimating the term $\int_{\R^N}\chi_2|u|^p\,dx$ in the same way as in \cite{BoHiLe2016} and using properties of $\chi_R$ and $E(u_0)<0$, we then obtain
$$
\frac{d}{dt}M_{\chi_R}[u] < 4s_1E(u_0).
$$
Thus the proof is completed.
\end{proof}


In the following, we are going to present some useful auxiliary results employed to establish Theorem \ref{blow-up vs global solt}.


\begin{lem}\label{lem A1}\cite[Lemma A.1]{BoHiLe2016}
Let $N\geq 1$ and $\chi$ be a real valued function such that $\nabla \chi \in W^{1,\infty}(\R^N)$. Then for any $u\in H^{\frac{1}{2}}(\R^N)$, there holds that
\begin{align*}
\left|\int_{\R^N}\bar{u}\nabla \chi \cdot \nabla u\,dx\right|\leq C\left(\||\nabla|^{\frac{1}{2}}u\|_2^2+\|u\|_2\||\nabla|^{\frac{1}{2}}u\|_2\right),
\end{align*}
where the constant $C>0$ depends only on $N$ and $\|\nabla \chi\|_{W^{1,\infty}}$.
\end{lem}


\begin{lem}\label{lem bup}
Let $N \geq 2$, $\frac 1 2 <s_2<s_1<1$ and $p \geq 2 + \frac{4s_1}{N}$. Let $u_0\in H_{rad}^{s_1}(\R^N)$ be such that $E(u_0)\neq 0$ and $u\in C_{T}(H_{rad}^{s_1}(\R^N))$ be the maximal solution of \eqref{evolv pb0} with initial datum $u_0$. If there exist $R>0$, $t_0>0$ and $C>0$ such that
\begin{align}\label{ineq bup}
M_{\chi_R}[u(t)]\leq -C\int_{t_0}^t\left(\|(-\Delta)^{\frac{s_1}{2}}u(\tau)\|_2+\|(-\Delta)^{\frac{s_2}{2}}u(\tau)\|_2\right)^2\,d\tau,
\end{align}
holds for any $t\geq t_0$. Then $u(t)$ cannot exist globally in time, i.e. $T<+\infty$.
\end{lem}
\begin{proof}
In light of Lemma \ref{lem A1}, the definition of $M_{\chi_R}[u]$ and the conservation of mass, we first get that
$$
|M_{\chi_R}[u(t)]|\leq C\left(\||\nabla|^{\frac{1}{2}}u\|_2^2+\||\nabla|^{\frac{1}{2}}u\|_2\right).
$$
Due to $s_1>\frac{1}{2}$, then the following interpolation estimate holds,
$$
\||\nabla|^{\frac{1}{2}}u\|_2\leq \|(-\Delta)^{\frac{s_1}{2}}u\|_2^{\frac{1}{2s_1}}\|u\|_2^{1-\frac{1}{2s_1}}.
$$
Thus
\begin{align}\label{MR inf frac u}
|M_{\chi_R}[u(t)]|\leq C\left(\|(-\Delta)^{\frac{s_1}{2}}u(t)\|_2^{\frac{1}{s_1}}+\|(-\Delta)^{\frac{s_1}{2}}u(t)\|_2^{\frac{1}{2s_1}}\right).
\end{align}
On the other hand, we claim that
\begin{align}\label{bound below}
\|(-\Delta)^{\frac{s_1}{2}}u(t)\|_2+\|(-\Delta)^{\frac{s_2}{2}}u(t)\|_2 \gtrsim 1, \quad \forall \,\, t \in [0, T).
\end{align}
Indeed, suppose that there exists a sequence of time $\{t_n\} \subset \R^+$ such that
$$
\|(-\Delta)^{\frac{s_1}{2}}u(t_n)\|_2+\|(-\Delta)^{\frac{s_2}{2}}u(t_n)\|_2=o_n(1).
$$
From Gagliardo-Nirenberg inequality \eqref{gn}, one then obtains that$\|u(t_n)\|_p=o_n(1)$. Hence $E(u(t_n))=o_n(1)$. This contradicts $E(u(t_n))=E(u_0) \neq 0$. Thus the claim follows. Combining \eqref{MR inf frac u} and \eqref{bound below} then implies that
\begin{align}\label{MR inf s}
|M_{\chi_R}[u(t)]|\leq C\left(\|(-\Delta)^{\frac{s_1}{2}}u(t)\|_2+\|(-\Delta)^{\frac{s_2}{2}}u(t)\|_2\right)^{\frac{1}{s_1}}.
\end{align}
Therefore, from the assumption \eqref{ineq bup}, we deduce that
$$
M_{\chi_R}[u(t)]\leq -C\int_{t_0}^t|M_{\psi_R}[u(\tau)]|^{2s_1}\,d\tau, \quad \forall \,\, t \geq t_0.
$$
By straightforward calculations, we then find that
$$
M_{\chi_R}[u(t)] \leq -C|t-t_1|^{1-2s_1}
$$
for some $0<t_1<\infty$. Consequently, we have that $M_{\chi_R}[u(t)] \rightarrow -\infty$ as $t \to t_1$, which implies that $u(t)$ cannot be global and then $T<+\infty$. This completes the proof.
\end{proof}


\begin{lem}\label{invariant condits}
If $s_c>0$, then the following conditions are invariant under the flow of \eqref{evolv pb0}.
\begin{enumerate}
\item [$(\textnormal{i})$] \eqref{energ u0 inf energ gs} and \eqref{mass u0 sup mass gs}.
\item [$(\textnormal{ii})$] \eqref{energ u0 inf energ gs} and \eqref{mass u0 sup mass gs1}.
\end{enumerate}
\end{lem}
\begin{proof}
From the conservation laws, it follows that \eqref{energ u0 inf energ gs} is invariant under the flow of \eqref{evolv pb0}. Next we shall prove that \eqref{mass u0 sup mass gs} and \eqref{mass u0 sup mass gs1} are invariant under the flow of \eqref{evolv pb0}. In view of Gagliardo-Nirenberg inequality \eqref{gn}, we first have that
\begin{align} \label{invariant1}
\begin{split}
\hspace{-1cm}E(u(t))M(u(t))^{\sigma_c}&=\frac 12 \left(\|(-\Delta)^{\frac{s_1}{2}}u(t)\|_2^2+\|(-\Delta)^{\frac{s_2}{2}}u(t)\|_2^2\right) \|u(t)\|_2^{2\sigma_c}-\frac 1 p \|u(t)\|_p^p\|u(t)\|_2^{2\sigma_c} \\
& \geq \frac 12 \left(\|(-\Delta)^{\frac{s_1}{2}}u(t)\|_2  \|u(t)\|_2^{\sigma_c}\right)^2 -\frac {C_{N,p,s_1}}{p} \left(\|(-\Delta)^{\frac{s_1}{2}}u(t)\|_2  \|u(t)\|_2^{\sigma_c}\right)^{\frac{N(p-2)}{2s_1}} \\
&=:f(\|(-\Delta)^{\frac{s_1}{2}}u(t)\|_2  \|u(t)\|_2^{\sigma_c}),
\end{split}
\end{align}
where $f : \R^+ \to \R$ is defined by
\begin{align} \label{deff}
f(x):=\frac 12 x^2 -\frac{C_{N,p,s_1}}{p}x^{\frac{N(p-2)}{2s_1}}.
\end{align}
Direct computations and Lemma \ref{pohoz 0} show that $f$ has a unique critical point
\begin{align} \label{critical}
x_0:=\left(\frac{2ps_1}{C_{N,p,s_1}N(p-2)}\right)^{\frac{2s_1}{N(p-2)-4s_1}}=\|(-\Delta)^{\frac{s_1}{2}} \phi\|_2\|\phi\|_2^{\sigma_c}
\end{align}
and
$$
\max_{x>0} f(x)=f(x_0)=\frac{N(p-2)-4s_1}{2N(p-2)} \left(\frac{2ps_1}{C_{N,p,s_1}N(p-2)}\right)^{\frac{4s_1}{N(p-2)-4s_1}}=\mathcal{E}(\phi)M(\phi)^{\sigma_c}.
$$
Therefore, by the conservation laws and \eqref{energ u0 inf energ gs}, we have that
$$
f(\|(-\Delta)^{\frac{s_1}{2}}u(t)\|_2  \|u(t)\|_2^{\sigma_c}) < \mathcal{E}(\phi)M(\phi)^{\sigma_c}=f(\|(-\Delta)^{\frac{s_1}{2}} \phi\|_2\|\phi\|_2^{\sigma_c}).
$$
It then follows from \eqref{mass u0 sup mass gs} and \eqref{mass u0 sup mass gs1} along with continuity arguments that \eqref{mass u0 sup mass gs} that \eqref{mass u0 sup mass gs1} are invariant under the flow of \eqref{evolv pb0}.
Thus the proof is completed.
\end{proof}

\begin{proof}[Proof of Theorem \ref{blow-up vs global solt}]

Let $u\in C_T(H_{rad}^{s_1}(\R^N))$ be a maximal solution to \eqref{evolv pb0} with initial datum $u_0 \in H^{s_1}_{rad}(\R^N)$. As a consequence of the conservation laws and Lemma \ref{invariant condits}, then the assertion $(\textnormal{i})$ follows immediately. Next we shall prove the assertion $(\textnormal{ii})$. Let us first consider the case $E(u_0)<0$. In this case, by Lemma \ref{Virial}, we obtain that
\begin{align*}
\frac{d}{dt}M_{\chi_R}[u]
&\leq N(p-2)E(u_0)-(N(p-2)-4s_1) \left( \|(-\Delta)^{\frac{s_1}{2}}u\|_2^2+\|(-\Delta)^{\frac{s_2}{2}}u\|_2^2 \right)\\
&\quad +C\left(R^{-2s_2}+R^{-\frac{(p-2)(N-1)}{2}+\varepsilon s_1}\|(-\Delta)^{\frac{s_1}{2}}u\|_{2}^{\frac{p-2}{2s_1}+\varepsilon}\right).
\end{align*}
Since $2<p<2+4s_1$, then we can choose $\varepsilon$ sufficiently small such that
$$
0<\frac{p-2}{2s_1}+\varepsilon<2.
$$
From the conservation laws and Gagliardo-Nirenberg inequality \eqref{gn}, we see that
$$
\|(-\Delta)^{\frac {s_1}{2}} u(t)\|_2+\|(-\Delta)^{\frac {s_2}{2}} u(t)\|_2 \gtrsim 1, \quad \forall \,\, t \in [0, T),
$$
Therefore, for any $R>0$ large enough, there holds that
\begin{align*}
\frac{d}{dt}M_{\chi_R}[u] \leq -\frac{N(p-2)-4s_1}{2} \left( \|(-\Delta)^{\frac{s_1}{2}}u\|_2^2+\|(-\Delta)^{\frac{s_2}{2}}u\|_2^2 \right).
\end{align*}
Suppose $T=+\infty$ and integrate above inequality on time, then there exists $t_0>0$ sufficiently large such that $M_{\chi_R}[u(t)] < 0$ for any $t\geq t_0$. Therefore, we derive that
$$
M_{\chi_R}[u(t)] \leq -\frac{N(p-2)-4s_1}{4}\int_{t_0}^{t}\left(\|(-\Delta)^{\frac{s_1}{2}}u(\tau)\|_{2}+\|(-\Delta)^{\frac{s_2}{2}}u\|_2^2\right)^2\,d\tau, \quad  \forall  \,\, t\geq t_0.
$$
In view of Lemma \ref{lem bup}, then the solution $u$ cannot exist for all time and $T$ must be finite.
Now we consider the case that $E(u_0) \geq 0$ and the assumptions \eqref{energ u0 inf energ gs} and \eqref{mass u0 sup mass gs1} hold. Note first that, by the Gagliardo- Nirenberg inequality \eqref{gn} and the conservation of mass, then
\begin{align*}
E(u(t)) &\geq \frac 12 \|(-\Delta)^{\frac{s_1}{2}}u(t)\|_{2}^2-\frac 1p \|u(t)\|^p_p \\
&\geq \frac 12 \|(-\Delta)^{\frac{s_1}{2}}u(t)\|_{2}^2-\frac{C_{N,p,s_1}}{p} \|u(t)\|_2^{p-\frac{N(p-2)}{2s_1}} \|(-\Delta)^{\frac{s_1}{2}}u\|_{2}^{\frac{N(p-2)}{2s_1}} \\
&=\frac 12 \|(-\Delta)^{\frac{s_1}{2}}u(t)\|_{2}^2-\frac{C_{N,p,s_1}}{p} \|u_0\|_2^{p-\frac{N(p-2)}{2s_1}} \|(-\Delta)^{\frac{s_1}{2}}u(t)\|_{2}^{\frac{N(p-2)}{2s_1}} \\
&=:g(\|(-\Delta)^{\frac{s_1}{2}}u(t)\|_{2}),
\end{align*}
where $g: \R^+ \to \R$ is defined by
$$
g(x):=\frac 12 x^2-\frac{C_{N,p,s_1}}{p} \|u_0\|_2^{p-\frac{N(p-2)}{2s_1}} x^{\frac{N(p-2)}{2s_1}}.
$$
In light of Lemma \ref{pohoz 0}, it is straightforward to compute that $g$ has a unique critical point
$$
x_1:=\left(\frac{2ps_1}{C_{N,p,s_1}N(p-2)}\right)^{\frac{2s_1}{N(p-2)-4s_1}} \|u_0\|_2^{-\frac{N(p-2)-2ps_1}{N(p-2)-4s_1}}=\|(-\Delta)^{\frac{s_1}{2}} \phi\|_2\|\phi\|_2^{\sigma_c}\|u_0\|_2^{-\sigma_c}
$$
and
\begin{align*}
\max_{x>0} g(x)=g(x_1)&=\frac{N(p-2)-4s_1}{2N(p-2)}\left(\frac{2ps_1}{C_{N,p,s_1}N(p-2)}\right)^{\frac{4s_1}{N(p-2)-4s_1}} \|u_0\|_2^{\frac{-2N(p-2)-4ps_1}{N(p-2)-4s_1}}\\
&=\mathcal{E}(\phi)M(\phi)^{\sigma_c}\|u_0\|_2^{-2\sigma_c}.
\end{align*}
From \eqref{energ u0 inf energ gs} and \eqref{mass u0 sup mass gs1}, we see that
$$
E(u_0)<g(x_1), \quad \|(-\Delta)^{\frac {s_1}{2}} u_0\|_2 >x_1.
$$
Therefore, by continuity arguments, we have that $\|(-\Delta)^{\frac {s_1}{2}} u(t)\|_2 >x_1$ for any $t \in [0, T)$. Let $0<\mu<1$ be such that
\begin{align}\label{Eu0 inf Ephi}
E(u_0)M(u_0)^{\sigma_c}< (1-\mu)\mathcal{E}(\phi)M(\phi)^{\sigma_c}.
\end{align}
This shows that
\begin{align*}
E(u_0)< (1-\mu)\mathcal{E}(\phi)M(\phi)^{\sigma_c}M(u_0)^{-\sigma_c}&=(1-\mu)\frac{N(p-2)-4s_1}{2N(p-2)}x_1^2 \\
&<(1-\mu)\frac{N(p-2)-4s_1}{2N(p-2)}\|(-\Delta)^{\frac {s_1}{2}} u(t)\|_2^2,
\end{align*}
from which we derive that
\begin{align}\label{last ineq}
(1-\mu)(N(p-2)-4s_1)\left(\|(-\Delta)^{\frac {s_1}{2}} u(t)\|_2^2+\|(-\Delta)^{\frac {s_2}{2}} u(t)\|_2^2\right)>2N(p-2)E(u_0).
\end{align}
Since $2<p<2+4s_1$, then we choose $\varepsilon>0$ small enough such that
$$
0<\frac{p-2}{2s_1}+\varepsilon<2.
$$
Note that
$$
\|(-\Delta)^{\frac {s_1}{2}} u(t)\|_2+\|(-\Delta)^{\frac {s_2}{2}} u(t)\|_2 \gtrsim 1, \quad \forall \,\, t \in [0, T).
$$
By Lemma \ref{Virial}, then
\begin{align*}
\frac{d}{dt}M_{\chi_R}[u(t)]
&\leq 2N(p-2)E(u_0)-(N(p-2)-4s_1)\left(\|(-\Delta)^{\frac {s_1}{2}} u(t)\|_2^2+\|(-\Delta)^{\frac {s_2}{2}} u(t)\|_2^2\right)\\
&+C\left(R^{-2s_2}+R^{-\frac{(p-2)(N-1)}{2}+\varepsilon s_1}\|(-\Delta)^{\frac{s_1}{2}}u(t)\|_{2}^{\frac{p-2}{2s_1}+\varepsilon}\right)\\
& \leq -\frac{\mu(N(p-2)-4s_1)}{2}\left(\|(-\Delta)^{\frac {s_1}{2}} u(t)\|_2^2+\|(-\Delta)^{\frac {s_2}{2}} u(t)\|_2^2\right) +o_R(1)\\
& \leq -\frac{\mu(N(p-2)-4s_1)}{4}\left(\|(-\Delta)^{\frac {s_1}{2}} u(t)\|_2+\|(-\Delta)^{\frac {s_2}{2}} u(t)\|_2\right)^2.
\end{align*}
Using Lemma \ref{lem bup}, we then have desired conclusion. Now we turn to prove the assertion $(\textnormal{iii})$.
In virtue of Lemma \ref{Virial}, we first have that
$$
\frac{d}{dt}M_{\chi_R}[u(t)]\leq 4s_1E(u_0), \quad \forall \,\, t\in [0,T).
$$
Suppose that $u(t)$ exists globally in time, then there exists a constant $t^*>0$ large such that
$$
M_{\chi_R}[u(t)] \leq -C t, \quad \forall \,\, t \geq t^*,
$$
This jointly with \eqref{MR inf s} then yields that
$$
\|(-\Delta)^{\frac {s_1}{2}} u(t)\|_2+\|(-\Delta)^{\frac {s_2}{2}} u(t)\|_2\geq Ct^{s_1}, \quad \forall \,\, t\geq t_*.
$$
Thus the proof is completed.
\end{proof}

\section{Orbital instability of ground state solutions} \label{section8}

In this section, we shall discuss orbital instability of ground state solutions to \eqref{fequ}-\eqref{mass} and present the proof of Theorem \ref{thm5}.

\begin{proof}[Proof of Theorem \ref{thm5}] Let $u_c \in S(c)$ be a ground state solution to \eqref{fequ}-\eqref{mass} at the level $\gamma(c)>0$. 
In view of Theorem \ref{thm6}, we may assume that $u_c$ is radially symmetric. Define
$$
\mathcal{Q}_c:=\{v \in S(c) : E(v) < E(u_c), Q(v)<0\}.
$$
Note first that $u_0:=(u_c)_{\tau} \in \mathcal{Q}_c$ for ant $\tau >1$, by Lemma \ref{monotonicity}. It then implies that $\mathcal{Q}_c \neq \emptyset$. Record that $u_0 \to u_c$ in $H^{s_1}(\R^N)$ as $\tau \to 1^+$. This suggests that $E(u_0) \to \gamma(c)$ as $\tau \to 1^+$. Let $u \in C([0, T), H_{rad}^{s_1}(\R^N))$ be the solution to \eqref{evolv pb0} with initial datum $u_0 \in H^{s_1}_{rad}(\R^N)$. In the following, we are going 
to demonstrate that $u(t)$ blows up in finite or infinite time. Observe first that $\mathcal{Q}_c$ is invariant under the flow of \eqref{evolv pb0}. Indeed, if not, by the conservation laws, then there exists $0<t_0<T$ such that $E(u(t_0))<E(u_c)$ and $Q(u(t_0))=0$. Hence
$$
\gamma(c) \leq E(u(t_0))<E(u_c).
$$
This is impossible, because of $E(u_c)=\gamma(c)$. For simplicity, we shall write $u=u(t)$. Due to $Q(u)<0$, it then follows from Lemma \ref{monotonicity} that there exists a constant $0<\tau_u<1$ such that $Q(u_{\tau_u})=0$. In addition, we know that the function $\tau \mapsto E(u_{\tau})$ is concave on $[\tau_u, 1]$. This then results in
$$
E(u_{\tau_u})-E(u) \leq (\tau_u-1) \frac{d}{d \tau} E(u_{\tau}) \mid_{\tau=1}=(\tau_u -1) Q(u).
$$
Since $Q(u)<0$ and $\gamma(c) \leq E(u_{\tau_u})$, by the conservation of energy, then
$$
Q(u) < (1-\tau_u) Q(u) \leq E(u)-E(u_{\tau_u}) \leq E(u)-\gamma(c)=E(u_0)-\gamma(c):=\delta<0,
$$
where $\delta>0$ is a constant.
If $2<p<2+4s_1$, then we can choose $\eps>0$ small enough such that
$$
0<\frac{p-2}{2s_1} +\eps<2.
$$
In addition, by the conservation laws $E(u)=E(u_0) \neq 0$, we have that
$$
\|(-\Delta)^{\frac {s_1}{2}} u\|_2^2+\|(-\Delta)^{\frac {s_2}{2}} u\|_2^2 \gtrsim 1.
$$
Using Lemma \ref{Virial}, Young's inequality and taking $R>0$ large enough, we then obtain that
\begin{align} \label{evolution}
\begin{split}
\frac{d}{dt}M_{\chi_R}[u]
&\leq 2N(p-2)E(u)-(N(p-2)-4s_1)\left(\|(-\Delta)^{\frac {s_1}{2}} u\|_2^2+\|(-\Delta)^{\frac {s_2}{2}} u\|_2^2\right)\\
&\quad +C\left(R^{-2s_2}+R^{-\frac{(p-2)(N-1)}{2}+\varepsilon s_1}\|(-\Delta)^{\frac{s_1}{2}}u\|_{2}^{\frac{p-2}{2s_1}+\varepsilon}\right)\\
& = 4Q(u)+C\left(R^{-2s_2}+R^{-\frac{(p-2)(N-1)}{2}+\varepsilon s_1}\|(-\Delta)^{\frac{s_1}{2}}u\|_{2}^{\frac{p-2}{2s_1}+\varepsilon}\right)\\
& \leq -\frac{\delta}{2}\left(\|(-\Delta)^{\frac {s_1}{2}} u\|_2^2+\|(-\Delta)^{\frac {s_2}{2}} u\|_2^2\right).
\end{split}
\end{align}
Thus, by Lemma \ref{lem bup}, we get that $u(t)$ cannot exist globally in time, i.e. $T<+\infty$. Let us now treat the case $p \geq 2+4s_1$. In this case, if $\|(-\Delta)^{\frac{s_1}{2}}u\|_{2}$ is unbounded, then $u(t)$ blows up in finite time or infinite time. If $\|(-\Delta)^{\frac{s_1}{2}}u\|_{2}$ is bounded,  namely $u(t)$ exists globally in time, by \eqref{evolution}, then for $R>0$ large enough,
$$
\frac{d}{dt}M_{\chi_R}[u] \leq -2\delta.
$$
Arguing as the proof of the assertion $(\textnormal{iii})$ of Theorem \ref{blow-up vs global solt}, we are able to achieve that there exists a constant $t^*>0$ large enough such that
\begin{align} \label{inst}
\|(-\Delta)^{\frac {s_1}{2}} u\|_2+\|(-\Delta)^{\frac {s_2}{2}} u\|_2\geq Ct^{s_1}, \quad \forall \,\, t\geq t_*.
\end{align}
Note that
$$
\|(-\Delta)^{\frac {s_2}{2}} u\|_2 \leq \|(-\Delta)^{\frac {s_1}{2}} u\|_2^{\frac{s_2}{s_1}} \|u\|_2^{1-\frac {s_2}{s_1}}=\|(-\Delta)^{\frac {s_1}{2}} u\|_2^{\frac{s_2}{s_1}} \|u_0\|_2^{1-\frac {s_2}{s_1}},
$$
where we used the conservation of mass. In virtue of \eqref{inst}, it then follows that $\|(-\Delta)^{\frac {s_1}{2}} u\|_2$ is unbounded. This contradicts with the assumption. Thus we have the desired conclusion and the proof is completed.
\end{proof}


\section{Appendix} \label{section9}


\begin{proof}[Proof of Lemma \ref{pohozaev}]
Utilizing scaling techniques, we only need to deduce Pohozaev identity of solutions to the following equation,
\begin{align} \label{zfequ11}
\left(\frac{k_{s_2}}{k_{s_1}}\right)^{\alpha s_1}(-\Delta)^{s_1} u +\left(\frac{k_{s_2}}{k_{s_1}}\right)^{\alpha s_2}(-\Delta)^{s_2} u + \lambda u=|u|^{p-2} u, \quad u \in H^{s_1}(\R^N),
\end{align}
where
$$
k_s=2^{1-2s} \frac{\Gamma(1-s)} {\Gamma(s)}, \quad \alpha=\frac{1}{s_1-s_2}.
$$
For this, by the harmonic extension theory from \cite{CS}, we are able to introduce the following extended problem,
\begin{align}\label{fequ111}
\left\{
\begin{aligned}
&-\mbox{div}(y^{1-2s_1} \nabla w+y^{1-2s_2} \nabla w)=0 \,\,\, &\mbox{in} \,\,\, \R^{N+1}_+, \\
&-\frac{\partial w}{\partial {\nu}}= k_{s_1, s_2} (|u|^{p-2}u-\lambda u) \,\,\, &\mbox{on} \,\, \R^N \times \{0\}.
\end{aligned}
\right.
\end{align}
where
$$
\frac{\partial w}{\partial {\nu}}:= \lim_{y \to 0^+} y^{1-2s_1} \frac{\partial w}{\partial y}(x, y)+y^{1-2s_2} \frac{\partial w}{\partial y}(x, y)  =- \frac{1}{k_{s_1}} (- \Delta )^{s_1} u(x)- \frac{1}{k_{s_2}} (- \Delta )^{s_2} u(x)
$$
and
$$
k_{s_1,s_2}:=\left(\frac{k_{s_1}^{s_2}}{k_{s_2}^{s_1}}\right)^{\alpha}.
$$
Multiplying \eqref{fequ111} by $(x, y) \cdot \nabla w$ and integrating on $\mathcal{B}^+(0, R)$, we get that
$$
\int_{\mathcal{B}^+(0, R)}\mbox{div}(y^{1-2s_1} \nabla w+y^{1-2s_2} \nabla w) ((x, y) \cdot \nabla w) \,dxdy=0,
$$
where
$$
\mathcal{B}^+(0, R):=\{(x, y) \in \R^{N+1}_+: |(x, y)| <R\}.
$$
Taking into account the divergence theorem, we find that
\begin{align} \label{ph1}
\begin{split}
&-\int_{\mathcal{B}^+(0, R)}(y^{1-2s_1} \nabla w+y^{1-2s_2} \nabla w)\cdot \nabla ((x, y) \cdot \nabla w)\,dxdy \\
&= k_{s_1, s_2} \int_{B(0, R)} (|u|^{p-2}u-\lambda u)(x \cdot \nabla u) \, dx -R \int_{\partial^+\mathcal{B}^+(0, R)} y^{1-2s_1} |\nabla w|^2 + y^{1-2s_2} |\nabla w|^2\,dS,
\end{split}
\end{align}
where $B(0, R):=\partial \mathcal{B}^+(0, R) \cap \R^N$ and $\partial^+\mathcal{B}^+(0, R):=\partial \mathcal{B}^+(0, R) \cap \R^{N+1}_+$. We are going to compute every term in \eqref{ph1}. Let us start with treating the first term in the right side hand of \eqref{ph1}. By the divergence theorem, then
\begin{align*}
k_{s_1, s_2} \int_{B(0, R)} (\lambda u -|u|^{p-2}u)(x \cdot \nabla u) \, dx&=\frac {\lambda k_{s_1, s_2}}{2} \int_{B(0, R)} x \cdot \nabla \left(|u|^2\right) \, dx -\frac {k_{s_1, s_2}}{p} \int_{B(0, R)} x \cdot \nabla \left(|u|^p\right) \, dx \\
&=-\frac{\lambda k_{s_1, s_2} N}{2} \int_{B(0, R)} |u|^2 \, dx +\frac {k_{s_1, s_2}N} {p} \int_{B(0, R)} |u|^p \, dx \\
&\quad +\frac{\lambda k_{s_1, s_2} R}{2} \int_{\partial B(0, R)} |u|^2\,dS- \frac {k_{s_1, s_2}R} {p} \int_{\partial B(0, R)} |u|^p \, dS.
\end{align*}
We next deal with the term in the left side hand of \eqref{ph1}. By the divergence theorem, then
\begin{align*}
&\int_{\mathcal{B}^+(0, R)} y^{1-2s_1} \nabla w \cdot \nabla ((x, y) \cdot \nabla w) \,dxdy\\
&=\frac 12 \int_{\mathcal{B}^+(0, R)} y^{1-2s_1} (x ,y) \cdot \nabla \left(|\nabla w|^2\right) \, dxdy
+\int_{\mathcal{B}^+(0, R)} y^{1-2s_1} |\nabla w|^2 \,dxdy\\
&=-\frac{N-2s_1}{2} \int_{\mathcal{B}^+(0, R)} y^{1-2s_1} |\nabla w|^2\,dxdy
+ \frac R 2 \int_{\partial^+\mathcal{B}^+(0, R)} y^{1-2s_1} |\nabla w|^2 \,dS.
 \end{align*}
Similarly, we can deduce that
\begin{align*}
&\int_{\mathcal{B}^+(0, R)}(y^{1-2s_2} \nabla w) \cdot \nabla ((x, y) \cdot \nabla w)\,dxdy\\
&=-\frac{N-2s_2}{2} \int_{\mathcal{B}^+(0, R)} y^{1-2s_2} |\nabla w|^2\,dxdy
+ \frac R 2 \int_{\partial^+\mathcal{B}^+(0, R)} y^{1-2s_2} |\nabla w|^2\,dS.
\end{align*}
Since $u \in H^{s_1}(\R^N)$ and $\nabla w \in L^2(y^{1-2s_1}, \R^{N+1}_+) \cap L^2(y^{1-2s_2}, \R^{N+1}_+)$, then there exists a sequence $\{R_n\} \subset \R$ such that
$$
R_n \left(\int_{\partial B(0, R_n)} |u|^2\,dS+ \int_{\partial B(0, R_n)} |u|^p \,dS \right)=o_n(1)
$$
and
$$
R_n \left(\int_{\partial^+\mathcal{B}^+(0, R_n)} y^{1-2s_1} |\nabla w|^2 \,dS+ \int_{\partial^+\mathcal{B}^+(0, R_n)} y^{1-2s_2} |\nabla w|^2 \,dS \right)=o_n(1).
$$
Making use of \eqref{ph1} with $R=R_n$ and taking $n \to \infty$, we then derive
\begin{align} \label{ph111}
\begin{split}
&\frac{N-2s_1}{2}\int_{\R^{N+1}_+} y^{1-2s_1} |\nabla w|^2 \, dxdy +\frac{N-2s_2}{2}\int_{\R^{N+1}_+} y^{1-2s_2} |\nabla w|^2\,dxdy \\
&=- \frac {\lambda k_{s_1, s_2} N} {2} \int_{\R^N} |u|^2 \,dx + \frac {k_{s_1, s_2}N} {p} \int_{\R^N} |u|^p\,dx.
\end{split}
\end{align}
On the other hand, multiplying \eqref{fequ111} by $w$ and integrating on $\R^{N+1}_+$, we obtain that
\begin{align} \label{ph112}
\hspace{-1cm}\int_{\R^{N+1}_+}y^{1-2s_1}|\nabla w|^2\,dxdy +\int_{\R^{N+1}_+}y^{1-2s_2}|\nabla w|^2\,dxdy
=-k_{s_1, s_2} \lambda \int_{\R^N} |u|^2\,dx + k_{s_1, s_2}\int_{\R^N} |u|^p\,dx.
\end{align}
Therefore, by combining \eqref{ph111} and \eqref{ph112}, we conclude that any solution $u \in H^{s_1}(\R^N)$ to \eqref{fequ111} satisfies the following identity,
$$
s_1\int_{\R^{N+1}_+}y^{1-2s_1}|\nabla w|^2\,dxdy +s_2 \int_{\R^{N+1}_+}y^{1-2s_2}|\nabla w|^2\,dxdy =\frac{k_{s_1, s_2}N(p-2)}{2p}\int_{\R^N}|u|^p\,dx.
$$
This completes the proof.
\end{proof}

\end{document}